\documentclass{gtart}
\usepackage{t1enc}
\usepackage[latin1]{inputenc}
\usepackage[english]{babel}
\usepackage{amssymb}
\usepackage{amsmath}
\usepackage{graphics}
\usepackage{hyperref}
 \usepackage{caption} 
\usepackage{amsthm, amssymb}
\usepackage{amsfonts}
\usepackage{epsfig,multicol}   

\usepackage{mathrsfs}

\theoremstyle{theorem}
\newtheorem{theorem}{Theorem}[section]
\newtheorem{algorithm}[theorem]{Algorithm}

\newtheorem{lemma}[theorem]{Lemma}
\newtheorem{corollary}[theorem]{Corollary}
\theoremstyle{definition}
\newtheorem{definition}[theorem]{Definition}

\newtheorem{remark}[theorem]{Remark}

\newtheorem{formula}[theorem]{}

\newtheorem{conjecture}[theorem]{Conjecture}
\newcommand{\arxiv}[1]{\href{http://arxiv.org/abs/#1}{\tt arXiv:\nolinkurl{#1}}}
\newcommand{\doi}[1]{\href{http://dx.doi.org/#1}{{\tt DOI:#1}}}

\newcommand{\googlebooks}[1]{(preview at \href{http://books.google.com/books?id=#1}{google books})}

\def\<{\langle}
\def\>{\rangle}

\begin{document}

\def\hpic #1 #2 {\mbox{$\begin{array}[c]{l} \epsfig{file=#1,height=#2}
\end{array}$}}
 
\def\vpic #1 #2 {\mbox{$\begin{array}[c]{l} \epsfig{file=#1,width=#2}
\end{array}$}}

\title{The Brauer-Picard group of the Asaeda-Haagerup fusion categories}
\author{Pinhas~Grossman}
\address{
}%
\email{p.grossman@unsw.edu.au}

\author{Noah~Snyder}
\address{
}%
\email{nsnyder@math.columbia.edu}

\address{%
\rm URLs:\stdspace \tt \url{http://math.columbia.edu/~nsnyder} \\ \url{http://w3.impa.br/~pinhas}}


\begin{abstract}
We prove that the Brauer-Picard group of Morita autoequivalences of each of the three fusion categories which arise as an even part of the Asaeda-Haagerup subfactor or of its index 2 extension is the Klein four-group. We describe the 36 bimodule categories which occur in the full subgroupoid of the Brauer-Picard groupoid on these three fusion categories. We also classify all irreducible subfactors both of whose even parts are among these categories, of which there are 111 up to isomorphism of the planar algebra (76 up to duality).  Although we identify the entire Brauer-Picard group, there may be additional fusion categories in the groupoid.  We prove a partial classification of possible additional fusion categories Morita equivalent to the Asaeda-Haagerup fusion categories and make some conjectures about their existence.  This is the submitted version of arXiv:1202.4396.
 
\end{abstract}

\maketitle

\section{Introduction}
If $N \subset M$ is a finite index subfactor, then the fundamental bimodule $ {}_N M_M $ generates tensor categories of $N-N$ and $M-M$ bimodules, called the even parts of the subfactor, as well as a Morita equivalence (i.e. an invertible bimodule category) of $N-M$ bimodules between them. It is natural to ask: what is the full Morita equivalence class of the even parts, and what are all the invertible bimodule categories between tensor categories in this class? 

In general, not much can be said about this question. But in the case that $N \subset M$ has finite depth, the even parts are fusion categories, and a result known as Ocneanu rigidity says that the ``maximal atlas'' \cite{MR1865095} of Morita equivalences is a finite groupoid, called the Brauer-Picard groupoid \cite{MR2677836}. (In fact the Brauer-Picard groupoid is a $3$-groupoid, but in this paper we only consider the $1$-truncation.)

Intimately related to the Brauer-Picard groupoid of a fusion category $\mathscr{C}$ is the family of module categories over $\mathscr{C}$. Certainly any bimodule category is in particular a module category over each of its left and right arguments, but also to any simple module category $ {}_\mathscr{C} \mathscr{K}$, there is associated a dual category of module endofunctors $\mathscr{D} $ such that $ {}_\mathscr{C} \mathscr{K} {}_\mathscr{D}$ is a Morita equivalence.   

For the fusion category $Rep_G$ of finite-dimensional complex representations of a finite group $G$ the situation is well understood. Let $ H \subseteq \tilde{H}$ be a central extension by $\mathbb{C}^* $ of a subgroup $H \subseteq G$. Then $Rep_{\tilde{H}} $ is a module category over $Rep_G $, and every simple module category over $Rep_G$  is of this form \cite{MR1976459}. For this reason module categories are sometimes called quantum subgroups.

Another class of fusion categories for which the representation theory is known is the family coming from quantum $\mathfrak{su}_2$ at roots of unity, which is parametrized by the Dynkin diagrams of type $A_n$. Here the module categories corresponding to Dynkin diagrams of type $A-D-E$ which have the same Coxeter number as $A_n$ (see \cite{MR996454, OcnLect, MR1777347} for this result in subfactor language, and \cite{MR1936496, MR1976459, MR2046203} for the translation of these results into the language of fusion categories and module categories); the corresponding Morita equivalences are implemented by the Goodman-de la Harpe-Jones subfactors \cite{MR999799}. Ocneanu has also announced the classification of quantum subgroups of the fusion categories coming from quantum $\mathfrak{su}_3$ and $\mathfrak{su}_4$ \cite{MR1907188} (see \cite{MR2545609,MR2553429} for details in the $\mathfrak{su}_3$ case). However, besides for these families there are very few examples for which a complete classification of quantum subgroups is known.

Motivated by the classification of small-index subfactors, Asaeda and Haagerup constructed two subfactors, one with index $\frac{5+\sqrt{13}}{2} $, known as the Haagerup subfactor, and one with index $\frac{5+\sqrt{17}}{2} $, known as the Asaeda-Haagerup subfactor \cite{MR1686551}. They called these subfactors ``exotic'' as they did not appear to be related to any previously known mathematical objects. We call the even parts of these subfactors, as well as any Morita equivalent fusion categories, the Haagerup and Asaeda-Haagerup fusion categories, respectively. A basic question is to determine the representation theory (i.e. the quantum subgroups and the Brauer-Picard groupoid) of these fusion categories.

In \cite{1102.2631} we considered the Haagerup fusion categories, and in the present paper we consider the Asaeda-Haagerup fusion categories.  The Asaeda-Haagerup fusion categories turn out to have a much richer structure, and we have developed better techniques in order to understand this example.  We now briefly describe the results and compare the techniques needed in the two cases.

In the Haagerup case \cite{1102.2631}, we began with the Haagerup subfactor which gave a Morita equivalence between a fusion category $\mathscr{H}_1$ with commutative fusion rules, and a fusion category $\mathscr{H}_2 $ with noncommutative fusion rules. The category $\mathscr{H}_2$ contains an order $3$ invertible object $\alpha$. It turned out that the category of $(1+\alpha+\alpha^2)-(1+\alpha+\alpha^2) $ bimodules in $\mathscr{H}_2$ is a new category (although it has the same fusion rules as $\mathscr{H}_2$), which we called $\mathscr{H}_3$. This ended up being everything: there is a unique Morita equivalence between each not-necessarily-distinct pair of these three categories, and there are no other Morita equivalent categories. In particular, the Brauer-Picard group of Morita autoequivalences of each of the three categories is trivial. 

In the Asaeda-Haagerup case, we begin with the Asaeda-Haagerup subfactor, which gives a Morita equivalence between a fusion category with commutative fusion rules, which we call $\mathscr{AH}_1 $ and a fusion category with noncommutative fusion rules, which we call $\mathscr{AH}_2 $. This time the category $\mathscr{AH}_2$ contains an order $2$ invertible object $\alpha$, and again the category of $(1+\alpha)-(1+\alpha) $ bimodules in $\mathscr{AH}_2$ is a new category (although this time all three categories have different fusion rules), which we call $\mathscr{AH}_3$. However unlike in the Haagerup case, there are multiple invertible bimodule categories between each pair of these fusion categories.

The key to finding the rest of the groupoid is the existence of several additional small-index subfactors.  
Motivated by the study of quadrilaterals of factors, it was shown in \cite{MR2418197} and \cite{MR2812458} that there are
subfactors with indices one larger than the Haagerup and Asaeda-Haagerup subfactors, i.e. $\frac{7+\sqrt{13}}{2} $  and 
$\frac{7+\sqrt{17}}{2} $; we call these subfactors $H+1 $ and $AH+1 $. In the Haagerup case $H+1$ just implements the trivial 
autoequivalence of $\mathscr{H}_1$, so it does not give any new information about the groupoid (and indeed, in \cite{1102.2631} 
we gave a ``trivial'' construction of $H+1$ exploiting this fact). However in the Asaeda-Haagerup case, $AH+1$ gives a second Morita 
equivalence between $\mathscr{AH}_1 $ and $\mathscr{AH}_3 $, which immediately implies that the group of autoequivalences of each of 
the Asaeda-Haagerup fusion categories is non-trivial.

Moreover, it was conjectured in \cite{MR2812458} that the ``plus one'' construction can be iterated once more in the Asaeda-Haagerup case to find a subfactor $AH+2 $ with index $\frac{9+\sqrt{17}}{2} $. We verify the existence of $AH+2$ and show that it gives a new autoequivalence of $\mathscr{AH}_1 $ which is not in the groupoid generated by $AH $ and $AH+1$. Finally, we find the full group of Morita autoequivalences of each of the three Asaeda-Haagerup fusion categories:

\begin{theorem}
The Brauer-Picard group of Morita autoequivalences of each of the Asaeda-Haagerup fusion categories is $\mathbb{Z}/2\mathbb{Z} \times \mathbb{Z}/2\mathbb{Z}$.

\end{theorem}

This means that there are a total of $3 \times 3 \times 4 =36$ bimodule categories between the three Asaeda-Haagerup fusion categories. We compute the fusion rules for all of these bimodule categories, and classify all subfactors which realize them (up to isomorphism of the planar algebra). There are over 100 such subfactors and we compute all of their principal graphs.

Although we are able to identify the full Brauer-Picard group of Morita autoequivalences of these three Asaeda-Haagerup fusion categories, 
we do not identify the full Brauer-Picard groupoid.  In addition to the three fusion categories which we construct, 
there may be several additional fusion categories Morita equivalent to them.  
We make several conjectures concerning the full Brauer groupoid which we plan to address in future work. 
In particular, it appears as though the Asaeda-Haagerup fusion categories are Morita equivalent to several 
more symmetric fusion categories with $4$ invertible objects.  
We hope that these additional fusion categories will open the door to a more symmetric construction of the Asaeda-Haagerup subfactor, and to generalizations where $\mathbb{Z}/4\mathbb{Z}$ is replaced by other groups.

The nontriviality of the group of autoequivalences allows us to apply the recently developed extension theory of fusion categories of \cite{MR2677836}. In a follow-up paper with David Jordan we show that there are no nontrivial invertible objects in the Drinfeld center of the Asaeda-Haagerup categories; the extension theory then implies that there are $\mathbb{Z}/2\mathbb{Z} $-graded fusion categories associated to each of the nontrivial autoequivalences in the Brauer-Picard groupoid of which the zero-graded parts are the corresponding Asaeda-Haagerup fusion categories.  This gives $9$ new fusion categories, including one coming from $AH+2$ which has a self-dual object with the relatively small dimension $\sqrt{\frac{9+\sqrt{17}}{2}} \approx 2.56155281$.

The main new technique of this paper is to analyze the combinatorial structure at the level of fusion rings for the whole Brauer-Picard groupoid simultaneously.  In particular, instead of considering principal graphs one at a time (as we did in \cite{1102.2631}) we consider the richer structure of a fusion module (or nimrep), and instead of considering one fusion module at a time we consider the whole structure of all possible fusion bimodules and all possible rules for composition.  These techniques allow us to eliminate certain combinatorial possibilities which look fine on their own, but which cannot be made compatible with all the other Morita equivalences.

Two aspects of our analyis required heavy computation. First, verifying the existence of $AH+2$ requires two difficult computations.  We first find a concrete representation of the principal even part of $AH+1$ as connections on $4$-graphs, following the outline of Asaeda-Haagerup's construction of the Asaeda-Haagerup subfactor; then we construct $AH+2$ by checking the algebra object relations following Asaeda-Grossman.  Conceptually these are very close to the original calculations, but both calculations are more difficult and we used C++ and Mathematica for bookkeeping, and to handle arithmetic with algebraic numbers.  Second, to classify module categories over the Asaeda-Haagerup fusion categories, it was necessary to first classify fusion modules over the corresponding fusion rings. Additionally, we had to check the multiplicative compatibility of triples of fusion modules, i.e. whether there exists a map from the relative tensor product of two fusion modules to a given third module which is compatible with the various fusion ring actions, Frobenius reciprocity criteria, and Frobenius-Perron dimensions. This was done through an elaborate combinatorial search written in C++; we outline the basic ideas of this search in Section 5 below.

The organization of the paper is as follows:

In Section 2, we present some background material on fusion categories, subfactors, and connections which will be necessary for what follows.

In Section 3, we provide diagrammatic proofs of two theorems about algebra objects in fusion categories which we require in later sections.

In Section 4, we describe three Asaeda-Haagerup fusion categories and their relation to the subfactors $AH, AH+1, AH+2 $, and give their Grothendieck rings.  We also construct $AH+2$; since the calculations are tedious and very closely analogous to the arguments in \cite{MR1686551,MR2812458} we give a rapid sketch, and include the details in the supplementary note \textit{AHplus2.pdf} and a companion Mathematica notebook.

In Section 5, we describe the combinatorial data which you get when you decategorify the Brauer-Picard groupoid (several fusion rings, fusion bimodules between them, and rules for composition) and the strong compatibility conditions that they must satisfy (like Frobenius reciprocity); we explain how to compute representations of fusion rings; and we present the classification of fusion modules and fusion bimodules over the three Asaeda-Haagerup rings. We also introduce the property of multiplicative compatibility for triples of fusion bimodules and show how to compute all possible multiplications between fusion bimodules.

In Section 6, we find the Brauer-Picard group of the Asaeda-Haagerup fusion categories and describe the full subgroupoid of the Brauer-Picard group generated by $AH, AH+1, AH+2$. We also give several results about possible additional objects in the Brauer-Picard groupoid and make conjectures about their existence.
 
In Section 7, we classify all subfactors whose even parts are among these three Asaeda-Haagerup fusion categories. 

We also include with the paper the following supplementary data files in plain text, which may be found in the arXiv source: \textit{AH1Modules}, \textit{AH2Modules}, \textit{AH3Modules}, \textit{Bimodules}, \textit{BimoduleCompatibility}, and \textit{ModuleCompatibility}. These files contain the multiplication tables for the fusion modules and bimodules over the Asaeda-Haagerup fusion rings as well as the lists of compatible compositions of modules and bimodules. 

\textbf{Acknowledgements:}
The authors would like to thank Marta Asaeda for help computing the connection on the fundamental $4$-graph in $AH+1 $ and for useful conversations. The authors would like to thank Scott Morrison for pointing out that the dual graph of $AH+2$ must be the same as the principal graph. The authors would also like to thank Masaki Izumi, David Jordan, and Emily Peters for helpful conversations. Finally, the authors would like to thank the referee for a number of helpful comments and suggestions. The Fusion Atlas was helpful for drawing graphs and checking their Frobenius-Perron weights, although we do not use it in any proofs.   Pinhas Grossman was partially supported by NSF grant DMS-0801235, and by a postdoctoral fellowship at IMPA, Brazil, and Noah Snyder was supported by an NSF Postdoctoral Fellowship at Columbia University.

\section{Background}
\subsection{Fusion categories}

\begin{definition} \cite{MR2183279,MR1010160}.
 A fusion category over an algebraically closed field $k$ is a $k$-linear semisimple rigid monoidal category with finitely many simple objects, finite-dimensional morphism spaces and simple identity object. A conjugation on a $\mathbb{C}$-linear monoidal category is a contravariant involutive endofunctor with fixes objects and commutes with the tensor product. A C$^*$-tensor category is a rigid $\mathbb{C} $-linear semisimple monoidal category with conjugation $*$ such that: (a) every morphism space is a Banach space and (b) $\| f \circ g \| \leq \| f \|  \| g \|$ and  $\| f^* \circ f \| = \| f^* \| \| f \|$ for all composable morphisms $f,g$. A unitary fusion category is a fusion category which is also a C$^*$-tensor category. 
\end{definition}

Fusion categories are categorifications of rings, and there are corresponding categorified notions of module categories (left and right) and bimodule categories. For definitions, as well as the notion of tensor product of bimodule categories and invertibility, see \cite{MR1976459,MR2183279}. All module categories are assumed to be semisimple.

\begin{definition} \cite{MR1966524, MR2677836}.
 A Morita equivalence between two fusion categories $\mathscr{C}, \mathscr{D} $ is an invertible $\mathscr{C}-\mathscr{D}$ bimodule category. The Brauer-Picard groupoid of a fusion category $\mathscr{C} $ is the category whose objects are fusion categories which are Morita equivalent to $\mathscr{C}$ and whose morphisms are Morita equivalences between two such fusion categories (again considered up to equivalence as bimodule categories.)
\end{definition}

\begin{remark}
 The Brauer-Picard group is actually defined in \cite{MR2677836} as a $3$-groupoid. In this paper we only consider the $1$-truncation. 
\end{remark}

Every Morita equivalence is indecomposable both as a left and right module category \cite{MR2677836}. If ${}_{\mathscr{C}} \mathscr{K} {}_{\mathscr{D}} $ is a Morita equivalence, then $\mathscr{D} $ is the dual category (i.e. the category of module endofunctors) to the module catgeory  ${}_\mathscr{C} \mathscr{K}$.

\begin{definition}
 An algebra object in a monoidal category is an object $A$ together with morphisms $1 \rightarrow A$ and $A \otimes A \rightarrow A$ satisfying the usual unit and associativity axioms.
\end{definition}

Given an algebra object $A$ in a fusion category $\mathscr{C}$, one can define left and right modules objects over $A$ in $\mathscr{C}$. The category of right $A$-modules in $\mathscr{C}$ is a left module category over $\mathscr{C}$. 

\begin{definition}
We say that $A$ is semisimple if the category $A$-mod is semisimple.   Following \cite{MR3039775}, $A$ is separable if there is an $A$-mod-$A$ map $A \rightarrow A \otimes A$ splitting the multiplication map.  We call $A$ simple if it is semisimple and the category $A$-mod is indecomposable as a right $\mathscr{C}$-module category.  Finally, we say that $A$ is a division algebra if $A$ is simple as an $A$-module.  
\end{definition}

\begin{theorem} \label{thm:separable}
In characteristic $0$, any semisimple algebra is separable.  Hence, the category of bimodules $A-mod-A$ is semisimple.  Furthermore, if $A$ is a division algebra, then $\mathrm{Hom}_{A-mod-A} (A , A \otimes A)$ is $1$-dimensional.  Hence any non-zero bimodule map $A \rightarrow A \otimes A$ has the property that its composition with multiplication is non-zero.
\end{theorem}
\begin{proof}
The first two claims are \cite[Prop. 2.7]{MR3039775}.  By Frobenius reciprocity, we have that $\mathrm{Hom}_{A-mod-A} (A, A \otimes A) \cong \mathrm{Hom}_{A-mod} (A, A)$, hence if $A$ is a division algebra the space of bimodule maps $\mathrm{Hom}_{A-mod-A} (A, A \otimes A)$ is one-dimensional.  Since $A$ is separable, any such map is a multiple of the splitting of multiplication.
\end{proof}

Note that by Frobenius reciprocity for induction/restriction, a simple algebra is a division algebra if and only if $\mathrm{Hom}(1, A)$ is one-dimensional.

Let $A$ be a simple algebra in $\mathscr{C}$. The left module category of  right $A$-modules is also a right module category over the category of $A-A $ bimodules in $\mathscr{C} $, where the action is by the relative tensor product over the algebra object $A$. The resulting bimodule category is a Morita equivalence. 

\begin{definition}\cite{MR1976459}.
 Let ${}_{\mathscr{C}} \mathscr{M} $ be a left module category over a fusion category $\mathscr{C} $. The internal hom is a bifunctor (contravariant in the first variable and covariant in the second) from $\mathscr{M} \times \mathscr{M} \rightarrow \mathscr{C} $ such that for any objects $M_1,M_2 \in \mathscr{M} $ and  $X \in \mathscr{C} $, we have  $ Hom(X \otimes M_1, M_2 ) \cong Hom(X, \underline{Hom}(M_1,M_2 ) ) $. Similarly, one can define internal hom for right module categories.
\end{definition}

The internal endomorphisms of an object $M $, which is defined as $\underline{End}( M)=\underline{Hom}(M,M) $, is an algebra object in $\mathscr{C} $.

\begin{theorem} \cite{MR1976459}.
 Let $M $ be a simple object in a semisimple indecomposable module category ${}_{\mathscr{C}} \mathscr{M}  $ over a fusion category $ \mathscr{C}$. Then the category of module objects over $ \underline{End}( M) $ in $\mathscr{C}$ is equivalent to ${}_{\mathscr{C}} \mathscr{M}  $ as a module category, and $ \underline{End}(M) $ is a division algebra.
\end{theorem}

We can use the internal hom to give an explicit description of the inverse of a Morita equivalence ${}_\mathscr{C} \mathscr{K}_\mathscr{D}$.

\begin{lemma} \cite{MR2677836}
The inverse to a Morita equivalence ${}_\mathscr{C} \mathscr{K}_\mathscr{D}$ is the opposite category ${}_{\mathscr{D}} \mathscr{K}^{op} {}_{\mathscr{C}}$ (where the actions have been twisted by the duals).
\end{lemma}

\begin{definition}
If $M$ is an object in a Morita equivalence ${}_\mathscr{C} \mathscr{K}_\mathscr{D}$ we denote the same object thought of in ${}_\mathscr{D} \mathscr{K}^{op}_\mathscr{C}$ by $M^*$.  The reason behind this notation is that $\underline{Hom}(M,N)$ (where $M$ and $N$ are both in $\mathscr{K}$) should be thought of as $N \otimes M^*$ where $M$ is an object in $\mathscr{K}^{op}$.
\end{definition}

Given two fusion categories $\mathscr{C}$ and $\mathscr{D}$, and a Morita equivalence ${}_\mathscr{C} \mathscr{K}_\mathscr{D}$ with inverse ${}_{\mathscr{D}} \mathscr{K}^{op} {}_{\mathscr{C}}$ one can take tensor products in several ways.  For example, one can tensor an object in $\mathscr{C}$ by an object in $\mathscr{K}$, or an object in $\mathscr{K}$ by an object in $\mathscr{K}^{op}$.  This can be formalized using the notion of a $2$-category.  This $2$-category has two objects $A$ and $B$. The $1$-morphisms from $A$ to itself are the objects of $\mathscr{C}$, the $1$-morphisms from $A$ to $B$ are the objects of $\mathscr{K}$, the $1$-morphisms from $B$ to $A$ are the objects of $\mathscr{K}^{op}$, and the $1$-morphisms from $B$ to itself are the objects of $\mathscr{D}$.  The $2$-morphisms are the $1$-morphisms in $\mathscr{C}$, $\mathscr{D}$, $\mathscr{K}$, and $\mathscr{K}^{op}$ \cite{MR1966524}, \cite{EGNO}.

The above ideas can be made somewhat more explicit by thinking about algebra objects (which was M\"uger's original approach).  Let $A$ be a division algebra object in a fusion category $\mathscr{C}$ and consider the Morita equivalence ${}_\mathscr{C} mod-A _{A-mod-A} $.  The inverse category is then given by the category of left $A$-modules.

More generally, given a fusion category $\mathscr{C} $, we can consider the $2$-category of all bimodules over division algebra objects in $\mathscr{C} $ (see \cite{MR2075605}). Then every fusion category in the Morita equivalence class of $\mathscr{C} $ is realized as the category of endomorphisms of some object in this $2$-category.  Thus any such fusion category is the category of $A-A$-bimodules for some divison algebra object $A$.  (In fact there will generally be multiple choices, corresponding to Morita equivalent division algebra objects). Similarly, every Morita equivalence between two fusion categories in the Morita equivalence class is realized as the category of $A-B$ bimodules between two division algebras in $\mathscr{C} $, and the inverse category is the category of $B-A $ bimodules. If $A,B,C$ are division algebras in $\mathscr{C} $,
the relative tensor product of the categories of $A-B$ bimodules and $B-C $ bimodules over the fusion category of $B-B $ bimodules is given by the category of $A-C $ bimodules; on the level of objects this relative tensor product is just composition of $1$-morphisms in the $2$-category.

Furthermore, in our setting, this $2$-category has duals.  Since a $2$-category of bimodules over a collection of special Frobenius algebras always has duals \cite{MR2075605}, it is enough to show that any division algebra in a fusion category over $\mathbb{C}$ has a canonical special Frobenius algebra structure.  (For the definition of special Frobenius algebra see \cite{MR1966524}, where they are called ``canonical Frobenius algebras.'')

\begin{lemma}
A divison algebra object in a fusion category over a field of characteristic $0$ has a canonical special Frobenius algebra structure.
\end{lemma}
\begin{proof}
First note that by semisimplicity, there is a trace map $\varepsilon: A \rightarrow 1$ which provides a splitting for the unit map $k \rightarrow A$.  The composition of multiplication and trace gives a pairing $A \otimes A \rightarrow 1$, which in turn provides a map of $A$-modules $A \rightarrow A^*$.  Since $A$ is simple as an $A$-module, this map must be an isomorphism, which is one of the equivalent definitions of a Frobenius algebra (see \cite{MR2500035}).  The composition $A \rightarrow A \otimes A \rightarrow A$ is non-zero, by Theorem \ref{thm:separable}.  Hence, after rescaling the trace, we get that $A \rightarrow A \otimes A \rightarrow A$ is the identity map and $A$ is a special Frobenius algebra.
\end{proof}

\begin{corollary} \cite{MR2075605}
If $\mathscr{C}$ is a fusion category over $\mathbb{C}$, then the $2$-category of division algebras in $\mathscr{C}$, bimodules, and bimodule maps is rigid.  In particular, it satisfies Frobenius reciprocity.
\end{corollary}

Note that if $M$ and $N$ are $A$-modules, then the $\underline{Hom}_{\mathscr{C}}(M,N)$ is just $N \otimes M^*$ where this dual is taken in the above rigid $2$-category.  A consequence of rigidity is Frobenius reciprocity.  In particular, we have associativity of internal hom with the relative tensor product.  Two particular consequences will be important for us.  The first is the ``basic identity'' of \cite{EGNO} (note that this basic identity holds much more generally), and the second gives a compatibility between relative Deligne tensor product and internal hom.

\begin{lemma}\label{basid}
If $L,M,N$ are objects in a Morita equivalence  ${}_\mathscr{C} \mathscr{K}_\mathscr{D}$,
then $$\underline{Hom}_{\mathscr{C}}(M,L) \otimes N \cong  L \otimes \underline{Hom}_{\mathscr{D}}(M,N)  .$$
\end{lemma}

\begin{lemma} \label{inthomass}
 Let ${}_{\mathscr{A}} \mathscr{K} { }_{\mathscr{B}}$ and ${}_{\mathscr{B}} \mathscr{L} { }_{\mathscr{C}}$ be invertible bimodule categories over fusion categories $\mathscr{A}$, $\mathscr{B}$, and $\mathscr{C} $.
 
 Then for any objects  $K ,M \in \mathscr{K}$ and $L, N \in \mathscr{L} $
 $$\underline{\text{Hom} }_{\mathscr{A}}(K \boxtimes L, M \boxtimes N) \cong\underline{\text{Hom} }_{\mathscr{A}}(K,M \otimes \underline{\text{Hom}}_{\mathscr{B}}(L,N))$$
and similarly
$$\underline{\text{Hom} }_{\mathscr{C}}(K \boxtimes L, M \boxtimes N ) \cong \underline{\text{Hom} }_{\mathscr{C}}( L, \underline{\text{Hom}}_{\mathscr{B}}(K,M) \otimes N),$$
where the tensor product between objects in $\mathscr{K} $ and $\mathscr{L} $ refers
to the relative tensor product of bimodule categories over $\mathscr{B} $.

Moreover, we have $$Hom(K \boxtimes L, M \boxtimes N)  \cong Hom(\underline{\text{Hom} }_{\mathscr{B}}(K,M),\underline{\text{Hom} }_{\mathscr{B}} (N,L)) .$$
\end{lemma}

Finally we recall the notion of dimension in fusion and module categories.

\begin{definition}
 The Grothendieck ring of a fusion category $\mathscr{C}$ is the based ring defined on the free Abelian group with basis indexed by the simple objects of $\mathscr{C}$ whose multiplicative structure constants on basis elements are given by the fusion rules of the category. There is a unique homomorphism from the Grothendieck ring to the real numbers which is positive on all the basis elements, called the Frobenius-Perron dimension. 
\end{definition}

We denote the Frobenius-Perron dimension of an object $X$ in a fusion category by $dim(X)$ or $d(X)$. If $M$ is an object in a semisimple module category over a fusion category, we define $dim(M)=\sqrt{dim(\underline{End}(M) ) }$. If $M$ belongs to a bimodule category over two fusion categories, the left and right internal end have the same dimensions so there is an unambiguous dimension associated to $M$ (see Lemma \ref{lem:leftright}).

%
%
%
%
%
%

\subsection{Subfactors}

A (II$_1$) subfactor is a unital inclusion $N \subset M$ of II$_1$ factors. The subfactor is said to have \textit{finite index} if $M $ is a finitely generated projective module over $N$ \cite{MR696688, MR860811}. In this case, letting $\kappa = {}_N M {}_M  $ and $\bar {\kappa} = {}_M M {}_N $, the bimodules $\kappa \otimes_M \bar{\kappa} $ and $\bar{\kappa } \otimes_N \kappa$ generate C$^*$-tensor categories of $M-M$ and $N-N$ bimodules, respectively; these tensor categories are called the \textit{principal and dual even parts} of the subfactor. The subfactor is said to have \textit{finite depth} if each of the even parts has only finitely many simple objects up to isomorphism - in this case the even parts are unitary fusion categories, and the category of $N-M$ bimodules generated by the even parts and $\kappa $ is a Morita equivalence between them. We will often use sector notation, in which object in categories are represented by lowercase Greek letters, tensor product symbols are suppressed and $(\kappa, \lambda) := dim(Hom( \kappa,\lambda) )$.

\begin{definition} \cite{MR1257245, MR1444286}.
 A Q-system in a C$^*$-tensor category is an algebra object such that the multiplication map is a coisometry. A Q-system has a dimension which coincides with the Frobenius-Perron dimension of the algebra object in the case of a unitary fusion category.
\end{definition}

For a finite index subfactor as above, the bimodules $\kappa \otimes_M \bar{\kappa} $ and $\bar{\kappa } \otimes_N \kappa$ have natural Q-system structures. Moreover, any $Q$-system in a C$^*$ tensor category can be realized in this way from a subfactor \cite{MR1257245, MR1444286, MR1960417}. The index of the subfactor is the dimension of the Q-system $\kappa \otimes_M \bar{\kappa} $.

\begin{definition}
 Let $N \subset M$ be a subfactor with fundamental bimodule $\kappa = {}_N M {}_M  $, principal and dual even parts $\mathscr{C} $ and $\mathscr{D} $, and Morita equivalence ${}_{\mathscr{C} }  \mathscr{K} {}_{\mathscr{D}}$ generated by $\kappa$. The principal graph of the subfactor is the bipartite graph with even vertices indexed by  $\mathscr{C}$, odd vertices indexed by $\mathscr{K}$, and $(\xi \kappa, \lambda )$ edges connecting each even vertex  $\xi$ with each odd vertex  $ \lambda$. The dual graph is defined similarly but using the dual even part $\mathscr{D} $ (which acts on the right of $\kappa $) instead.   
\end{definition}

\subsection{Connections}

The theory of connections on graphs was introduced by Ocneanu \cite{MR996454}. A good resource is \cite{MR1642584}. See also \cite{MR1686551}. 

\begin{definition}
 By a $4$-graph we mean a graph with $4$ finite sets of vertices $V_i$, for $i \in \mathbb{Z}/4\mathbb{Z}$ and $4$ finite sets of edges $E_i$,  such that each edge $e$ in $E_i$ connects a vertex in $V_i $ (called the source $s(e)$) with a vertex in $V(i+1) $ (called the target, $t(e)$). A cell in a $4$ graph is a choice of edges $e_i \in E_i $ such that $s( e_{i+1}) = t( e_i)  $.  A connection on a $4$-graph is an assignment of a complex number to each cell.
\end{definition}

We think of the edges as being placed in a square, clockwise with $0$ at the top left. We then call $V_0-E_0-V_1 $ the top graph, $V_1-E_1-V_2$ the right graph, etc.

\begin{figure} 
\centering
\hpic{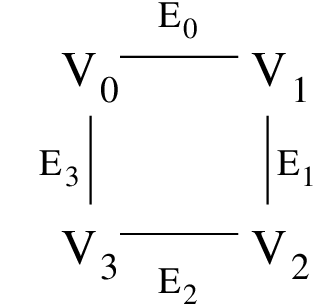} {0.8 in} 
\caption{Schematic representation of a $4$-graph.}
\end{figure}

 Given a finite depth subfactor $N \subset M$, one can define a flat, binuitary connection on a $4$-graph whose upper and left graphs are each the principal graph of the subfactor and whose lower and right graphs are each the dual graph of the subfactor. Moreover, any flat, binuitary connection on a $4$-graph arises this way.  

\begin{definition}
 The edge space $H^G_E$ of a $4$-graph $G$ is the Hilbert space with orthonormal basis indexed by the edges of the $4$-graph. For every pair of vertices $v \in V_i, w \in V_{i+1}$, we have a subspace $H^G_{v,w}  \subset H^G_E$ spanned by the edges $\{ e | s(e)=v, t(e) = w \} $. An edge space map between two $4$-graphs $G_1,G_2$ with the same vertex sets $\{V_i\}$ is a linear map $T: H^{G_1}_E \rightarrow H^{G_2}_E $ such that $T(H^{G_1}_{v,w}) \subseteq H^{G_2}_{v,w}, \forall v,w$. A gauge transformation on a $4$-graph $G$ is an edge space map from $G$ to itself. A vertical gauge transformation is a guage transformation which fixes all of the horizontal edge spaces. 
\end{definition}

A connection on a $4$-graph can be extended (linearly for the left and lower graphs, antiliearly for upper and right graphs) to cells consisting of vectors in the edge space. Then composing the connection with any gauge transformation gives a new connection.

In \cite{MR1686551} it was shown that any binuitary connection on a $4$-graph defines a bimodule over the II$_1$ factors generated by the string algebras of the upper and lower graphs. Moreover, sums and compositions of connections were defined which correspond to direct sums and tensor product of bimodules. 

\begin{theorem}\cite{MR1686551}
Two connections on $4$-graphs with the same connected horizontal graphs give isomorphic bimodules iff the vertical graphs are also the same and the connections are equivalent up to a unitary vertical gauge tranformation.
\end{theorem}

This gives a concrete realization of the unitary fusion categories coming from a subfactor: one starts with the flat connection on the principal and dual graphs, which corresponds to the fundamental bimodule $\kappa = {}_N M {}_M $. Then by decomposing powers of $\kappa $ and $\bar{\kappa } $ into direct sums of irreducible connections, one obtains the connections corresponding to the simple objects in the even parts of the subfactor (the ``even'' bimodules), as well as those corresponding to the simple objects in the Morita equivalence (the ``odd'' bimodules). Finally, morphisms in these categories, which are intertwiners of the bimodules, can be represented as edge space maps of the corresponding connection $4$-graphs. 

\section{Diagrams for algebra objects}

In this section we prove two lemmas which we will need later on.  Both use diagrammatic techniques.  The first lemma shows that certain objects automatically admit algebra structures, and the second allows us to characterize when $\kappa \bar{\kappa}$ and $\lambda \bar{\lambda}$ are isomorphic as algebra objects.

\subsection{Intertwiner diagrams}

Following \cite{MR0281657, MR1091619, MR1113284, math.QA/9909027}, we will often use diagrams for computations in tensor categories.  Morphisms are represented by vertices or boxes, from which emanate strings labeled by the source objects (upwards) and by the target objects (downwards). Straight strings labeled by objects correspond to identity morphisms, and strings labeled by identity objects are often suppressed. Tensoring is depicted by horizontal concatenation, and composition by vertical concatenation. Diagrams are read from top to bottom. Then various planar isotopies can be applied to the diagram according to the duality rules of the category.

Let $\xi$ be an object in a fusion category with left and right duals ${}^*\xi $ and $\xi^* $, respectively. Then by the rigidity property, there are morphism $\eta^r_{\xi}:1 \rightarrow \xi \otimes \xi^* $, $\epsilon^r_{\xi}: 1 \rightarrow \xi^* \otimes \xi$ such that $ (Id_{\xi} \otimes \epsilon^r_{\xi} ) \circ (\eta^r_{\xi} \otimes Id_{\xi} )=Id_{\xi}$. (Here we suppress the associativity and unit isomorphisms.) Similarly, there are morphisms  $\eta^l_{\xi}:1 \rightarrow {}^* \xi \otimes \xi $, $\epsilon^l_{\xi}: 1 \rightarrow \xi \otimes {}^*\xi$ such that $ ( \epsilon^l_{\xi} \otimes Id_{\xi} ) \circ ( Id_{\xi} \otimes \eta^l_{\xi}) = Id_{\xi}$.

Diagrammatically, we write:

$$ \vpic{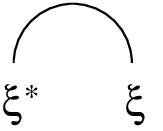} { 0.5in } = \eta^l_{\xi}, \vpic{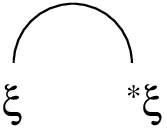} { 0.5in } = \eta^r_{\xi} , \vpic{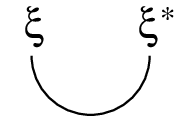} { 0.5in } = \epsilon^l_{\xi}, \vpic{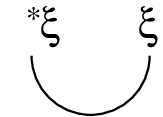} { 0.5in } = \epsilon^r_{\xi}    $$ 

Then the duality relations are expressed as:

$$\vpic{ 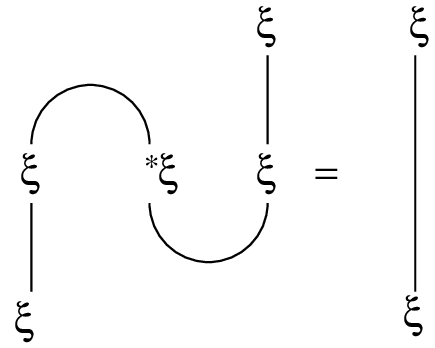} {1.2in}, \vpic{ 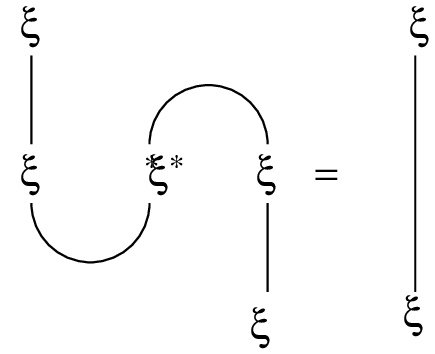} { 1.2 in} $$

If the fusion category is unitary, then ${}^*\xi \cong \xi^* $ and we write $\bar{\xi} $ for the dual object and reserve the ${}^*$ symbol for the unitary conjugation on morphisms. Let $\xi = \bar{\xi} $ be a self-dual object in a unitary fusion category. We may choose $\epsilon^r_{\xi} = (\eta^l)^*_{\xi}$ and   $\epsilon^l_{\xi} = (\eta^r)^*_{\xi}$, and then there is a number $c_{\xi} \in \{ \pm1\} $ such that $\eta^l_{\xi } = c \eta^r_{\xi} $. If $c = 1 $ we say that $\xi$ is symmetrically self-dual, and we write 
$$ \vpic{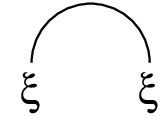} {0.5in} , \vpic{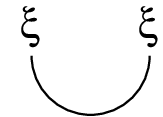} {0.5in} $$ for the common left and right duality maps. The invariant $c$ is multiplicative in the sense that if $\xi, \mu, \nu$ are self-dual objects in a unitary fusion category such that $(\xi, \mu \nu  ) \neq 0 $, then $c_{\xi} = c_{\mu}c_{\nu} $.

Finally we recall the following computation for Q-systems corresponding to $2$-supertransitive algebra objects.

\begin{lemma}\cite{MR2418197} \label{2stqs} Let $\sigma$ be a symmetrically self-dual simple object in a unitary fusion category with $d=dim(\sigma)\neq 1$. Fix a duality map $\hpic{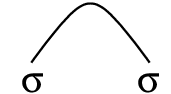} {0.3in} $. Then $1+\sigma$ admits a Q-system if there is an isometry $d^{-\frac{1}{4}}  \hpic{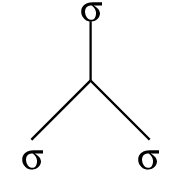} {0.5in} $ such that the following relations are satisfied:

(1) $\hpic{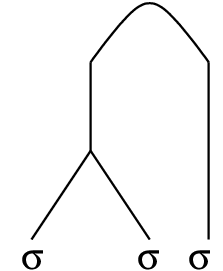} {0.8in} = \hpic{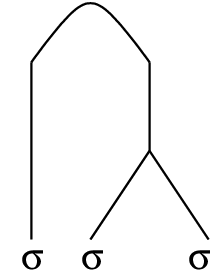} {0.8in} $\\\\ 

(2) $\displaystyle \frac{\sqrt{d+1}}{d} \left(\hpic{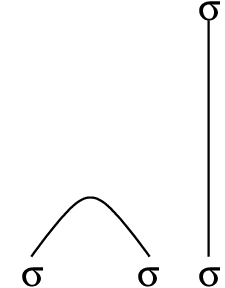} {0.8in} - \hpic{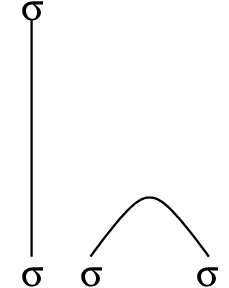} {0.8in} \right) =  \hpic{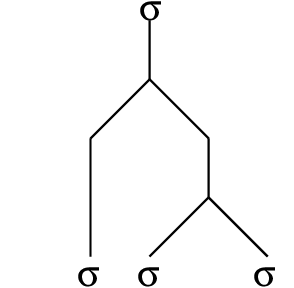} {0.8in} - \hpic{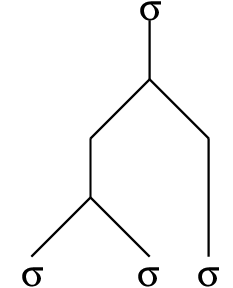} {0.8in} $.

\end{lemma}

\subsection{$4$-supertransitive objects admit Q-systems }
In this section we prove that if $\xi^2  \cong 1 + \xi + \mu $, where $\mu $ is a simple object with $d(\mu ) > 1 $, then $1+\xi$ admits an algebra structure.  This is a Wenzl-style recognition theorem and our argument is inspired by similar results in \cite{MR1237835,MR2132671,MR2783128}.

Let $\xi $ be a symmetrically self-dual object in a unitary fusion category. Then the rotation operator $$\vpic{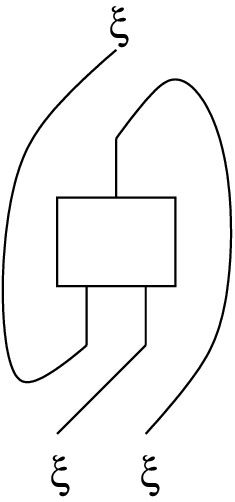 } {0.8in} $$ acts as a period $3$ automorphism on the vector space $Hom(\xi, \xi^2) $. If $(\xi, \xi^2)=1$ then the automorphism is necessarily scalar multiplication by a cube root of unity; this cube root of unity is called the rotational eigenvalue of $\xi$.

\begin{lemma}
 Let $\xi $ be an object in a unitary fusion category such that $\xi^2  \cong 1 + \xi + \mu $, where $\mu $ is a simple object with $d=d(\mu ) > 1 $. Then $\xi $ is symmetrically self-dual with a rotational eigenvalue of $1$.
\end{lemma}
\begin{proof}
 Since $(\xi,\xi^2 )=1 $, we have $c_{\xi}=c_{\xi}^2 $, so $c_{\xi}=1 $ and $\xi $ is symmetrically self-dual. Fix a duality map $\eta_{\xi}$  and an isometry $$\vpic{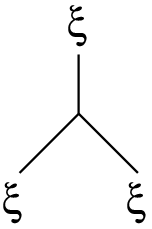} {0.5in} =v \in Hom(\xi,\xi^2 )$$ Let  $$\vpic{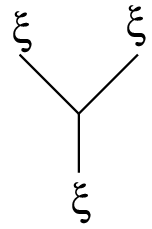} {0.5in } = v^* .$$ Then the following projections form an orthognal basis of $End(\xi^2)$:
 $$e_1 = \frac{1}{d} \vpic{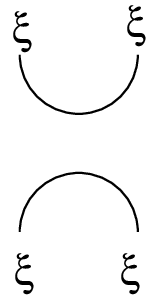} {0.5in} ,\hspace{.3in}  e_2 = \vpic{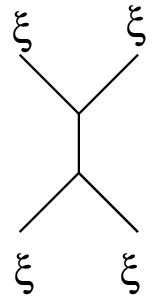} {0.5in} , \hspace{.1in} \text{and} \hspace{.1in} e_3 = 1 - e_1-e_2.$$

Assume for the sake of contradiction that the rotational eigenvalue is not $1$.  Consider the diagram  $$\vpic{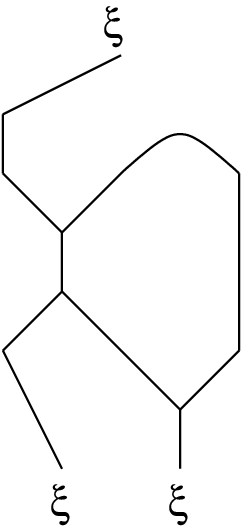} {0.5in} .$$  Note that this diagram is rotationally invariant (since any rotation picks up a cube of the rotational eigenvalue).  Therefore, since $(\xi^2, \xi) = 1$, we see that it must be zero.   Let  $$x = \vpic{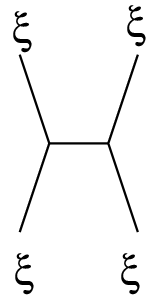} {0.5in} := \vpic{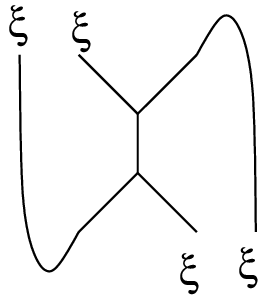 } {0.8 in} \in End(\xi^2).$$  
As we saw a moment ago $xe_2=0 $. Let $tr: End(\xi^2 )\rightarrow \mathbb{C}$ be the coefficient of the identity under the action of the linear operator $$\vpic{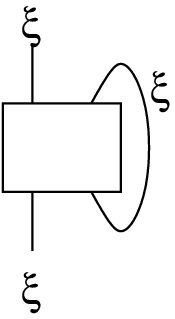 } {0.5in} .$$ Then $tr(e_1)=\frac{1}{d}, tr(e_2)=1, tr(e_3)=d-1-\frac{1}{d}  $.  Also we have $xe_1=e_1 $ and  $tr(x)=0$. Write $x = e_1+be_3$. Then we have $0 = tr(x)=\frac{1}{d}+b(d-1-\frac{1}{d} )$. Solving for $b$ and gathering terms, we get the linear relation   
 $$\vpic{ xiprojs } {0.5in} -  \frac{1}{d^2-d-1} \vpic{xiisoproj } {0.5in } = \frac{d-1}{d^2-d-1} \vpic{xijones } {0.5in} - \frac{1}{d^2-d-1} \vpic{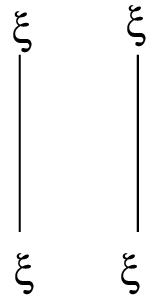} {0.5in} .$$ The rotation operator $ \vpic{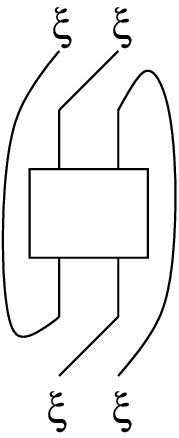 } {0.5in} $ acts on this linear relation by permuting the two diagrams on the left and permuting the two diagrams on the right, and multiplying the rotation by $1+d-d^2 $ gives

$$a \vpic{ xiprojs } {0.5in} -(d^2-d-1)  \vpic{xiisoproj } {0.5in } =  \vpic{xijones } {0.5in} -(d-1 ) \vpic{xixiid} {0.5in} .$$

 By the linear independence of $e_1,e_2,e_3 $ this is only possible if $d=1$.   This gives a contradiction; thus the rotational eigenvalue must be $1$.
\end{proof}

\begin{remark}
Note that the assumption that $d(\mu ) > 1$ is necessary - if $d(\mu ) = 1$ there are three fusion categories with those fusion rules coming from twisting $\mathrm{Rep}_{S_3}$ in different ways, and the three fusion categories yield three different rotational eigenvalues.
\end{remark}

\begin{theorem}\label{4stqs}
 Let  $\xi $ be an object in a unitary fusion category such that $\xi^2  \cong 1 + \xi + \mu $, where $\mu $ is a simple object with $d(\mu ) > 1 $. Then $1+\xi $ admits a Q-system. 
\end{theorem}
\begin{proof}
 Using the same notation as in the previous proof,  we write \\ $x=e_1+ae_2+ be_3$ and use the rotation and linear independence to solve for $b=\frac{1}{1+d}$, $ a=b\pm 1 $. Taking traces of the equation determines the sign, and we get $$\vpic{ xiprojs } {0.5in} - \vpic{xiisoproj } {0.5in } = \frac{1}{d+1}\left( \vpic{xijones } {0.5in} + \vpic{xixiid} {0.5in} \right) .$$ Using this and the rotational invariance of $\vpic{xi_iso} {0.5in} $ the Q-system relations for $1+\xi$ can be easily verified using Lemma \ref{2stqs}. 
\end{proof}
 
%

\subsection{Algebra isomorphisms come from invertible objects in the dual category}
   Let $\kappa$ be an object in an invertible bimodule category ${}_{\mathscr{C}} \mathscr{K} {}_{\mathscr{D}} $ over two fusion categories. Then there is an object $\bar{\kappa}$ in the inverse category ${}_{\mathscr{D}} \bar{\mathscr{K}} {}_{\mathscr{C}}$ such that there are unit maps $\eta_{\kappa}: 1 \rightarrow \kappa \otimes_{\mathscr{D}} \bar{\kappa} $ and $\eta_{\bar{\kappa}}: 1 \rightarrow \bar{\kappa} \otimes_{\mathscr{C}} \kappa  $ along with co-unit maps satisfying the usual duality relations, which together with the multiplication maps

$$ \vpic{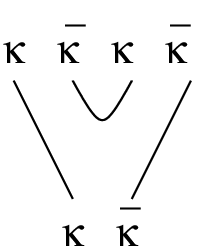} {0.75in} , \vpic{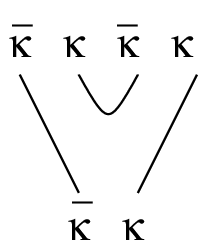} {0.75in} $$ 

form the internal end algebras of $\kappa $ in $\mathscr{C}, \mathscr{D} $, respectively. 

\begin{definition}
 An algebra isomorphsim between two algebra objects $\xi, \eta$ in a fusion category is an isomorphism between $\xi $ and $ \eta$ which commutes with the multiplication and unit maps on $\xi, \eta $.
\end{definition}

If we express the algebra objects as $\xi = \kappa \bar{\kappa}, \eta = \lambda \bar{\lambda} $ for two module objects and write the isomorphism as $$ \vpic{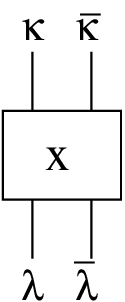} {0.45in} ,$$ the algebra isomorphism conditions become $$\vpic{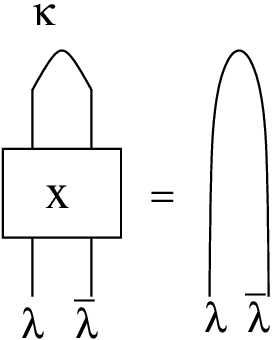} {1in} \text{ and } \vpic{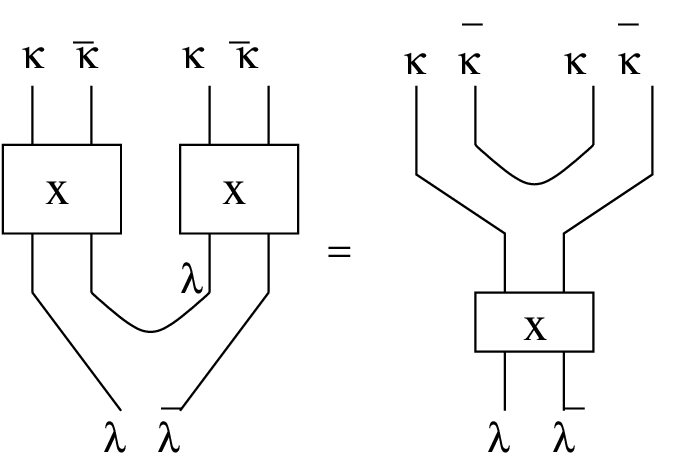 } {1.8in} .$$

\begin{theorem}\label{alg_iso}
 Let $\xi = \kappa \bar{\kappa}$ and $\eta = \lambda \bar{\lambda}  $ be isomorphic algebra objects in a fusion category $\mathscr{C}$, where $\kappa $ and $\lambda $ are objects in a left module category ${}_{\mathscr{C}} \mathscr{K} $. Then there is an invertible object $\alpha$ in the dual category $\mathscr{D}$ such that $\kappa \alpha \cong \lambda$.
\end{theorem}

\begin{proof}
 Let $$\vpic{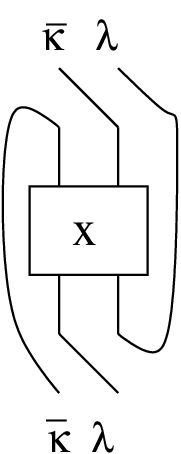} {0.6in} $$ be the image of the algebra isomorphism under rotation, and let $d = dim(\kappa) =dim(\lambda)$. Then from the algebra isomorphism multiplication relation we obtain: $$\vpic{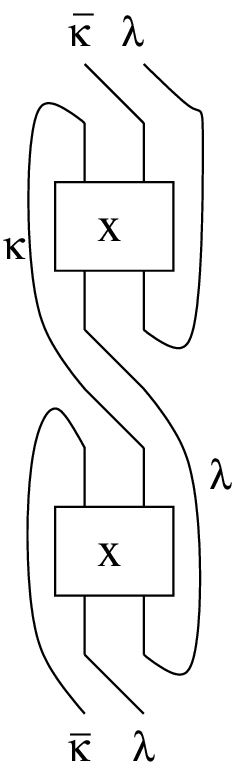} {0.5in} = d  \vpic{klrot} {0.5in} .$$  Thus we see that $$\frac{1}{d} \vpic{klrot} {0.5in} $$ is an idempotent. Call the object which is the image of the idempotent $\alpha $. From the algebra isomorphism unit relation we get: $$ dim(\alpha ) = \frac{1}{d} \vpic{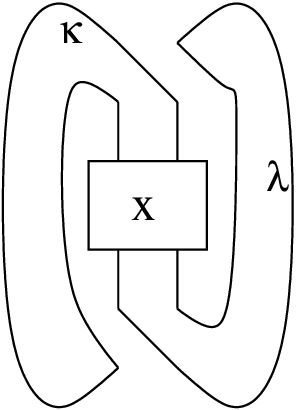} {0.75in} = 1.$$ So $\alpha $ is invertible. Let $$\vpic{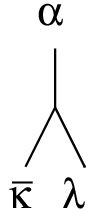 } {0.3in} $$ be an inclusion map. Then $$\vpic{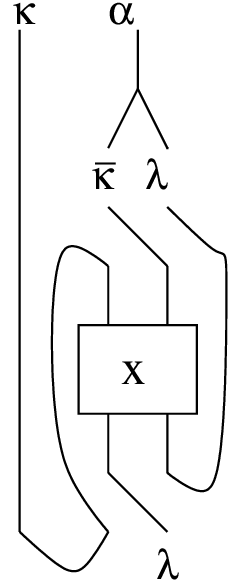 } {0.7in} $$ is the desired isomorphism from $ \kappa \alpha$ to $\lambda $.

\end{proof}

\section{The Asaeda-Haagerup categories}
In this section we recall some results about the Asaeda-Haagerup subfactor \cite{MR1686551}, and the related AH+1 subfactor \cite{MR2812458}.  In addition we sketch the construction of a new subfactor, the AH+2 subfactor, using the techniques from those two papers.

\subsection{AH and AH+1}
The Asaeda-Haagerup subfactor, constructed in \cite{MR1686551}, is a finite depth subfactor with index $\frac{5+\sqrt{17}}{2} $  and principal graphs $$\hpic{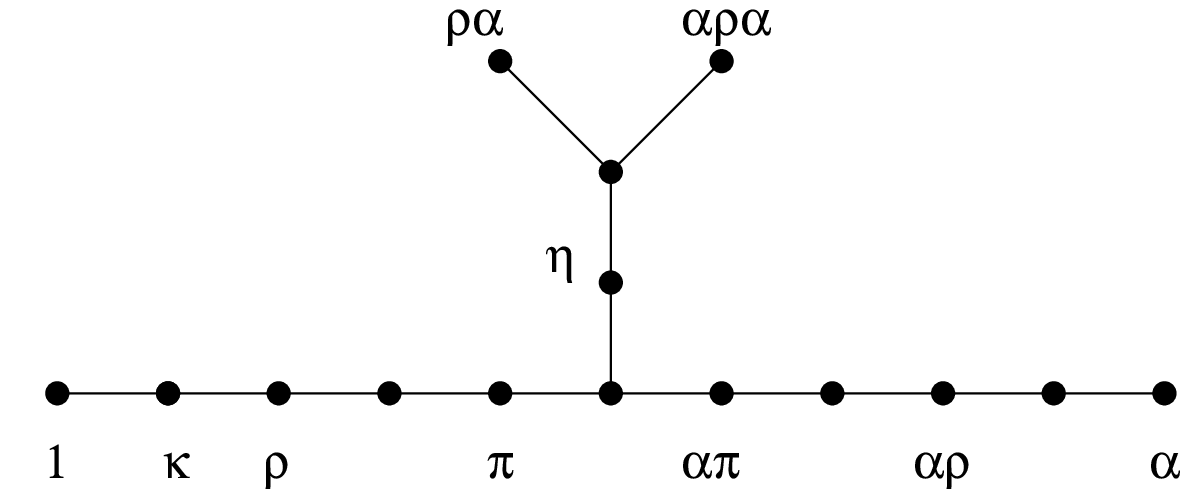} {1 in} \text{ and }\hpic{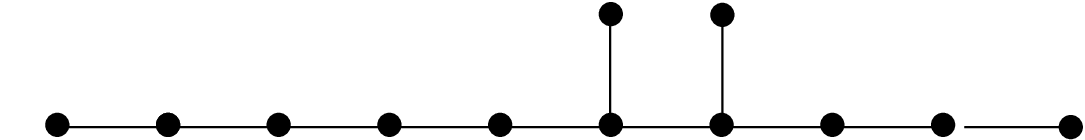} {0.25in} .$$

Here we label the simple objects in the principal even part as well as the generating module object $\kappa$. As part of their construction of the subfactor, Asaeda and Haagerup explicitly wrote down the unique (up to gauge choice) connection for $\kappa$, the connection $\rho$ defined (up to vertical gauge choice) by $1 + \rho \cong \kappa \bar{\kappa }$, and the connection for the automorphism $\alpha$ (which is again uniquely determined up to vertical gauge choice, and can be taken to be identically $1$.) Then they wrote down a vertical gauge transformation (i.e. an intertwiner) between $\rho \alpha \kappa $ and $\alpha \rho \alpha \kappa $.

Motivated by the study of quadrilaterals of factors, Izumi conjectured the existence of a Q-system for the object $1+ \bar{\kappa} \alpha \kappa $ in the dual even part. This conjecture was verified in \cite{MR2812458} by showing that the following diagrammatic relations hold, and that they lead a solution to the Q-system equations \ref{2stqs} for $1+\bar{\kappa} \alpha \kappa$:

\begin{formula} \label{algebrarelations}
The Asaeda-Haagerup algebra relations:

(a) $\hpic{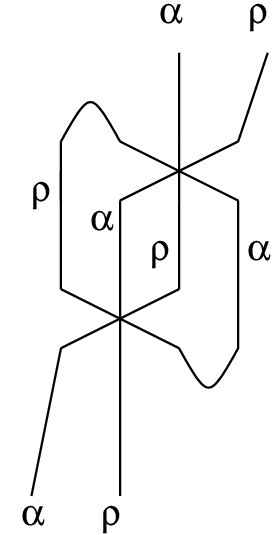} {1in} = c Id_{\alpha \rho}$, $\hpic{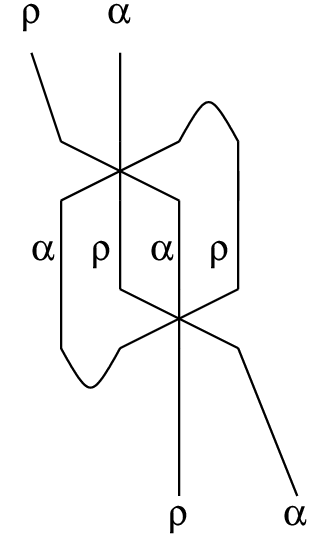} {1in} =c Id_{\rho \alpha}$ \\

(b) $\hpic{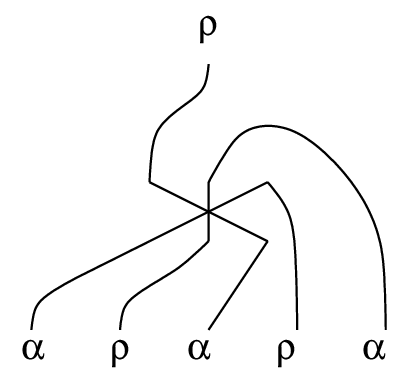} {0.88in} = \hpic{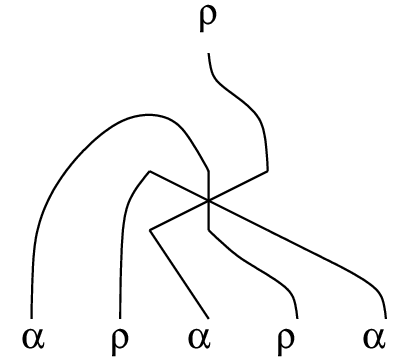} {0.88in} $\\

(c) $\hpic{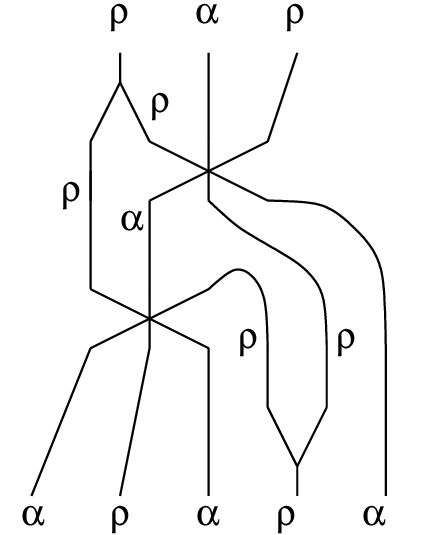} {1.2in} = \hpic{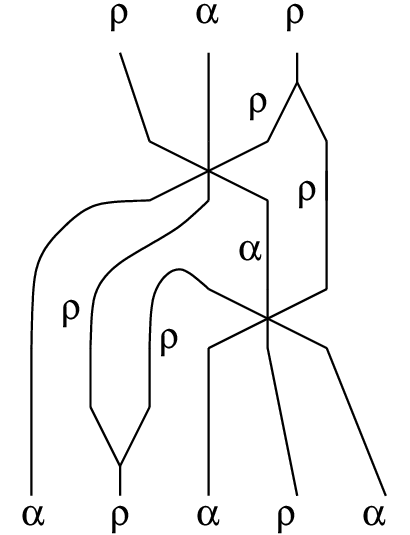} {1.2in} $\\

 (d) $\hpic{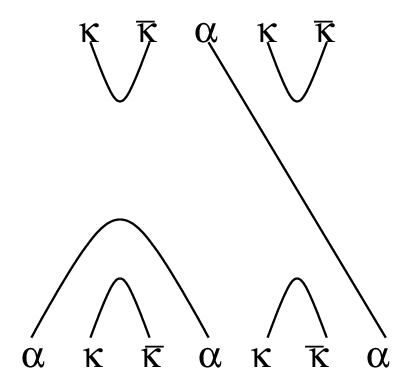} {1in} = \hpic{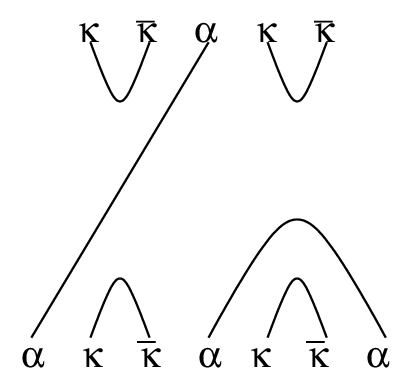} {1in} $.

\end{formula}

Note that although the vertices of the intertwiner diagrams are each only determined up to a scalar, the relations make sense independently of a choice of scalars, since the same vertices appear on both sides of each equation (except for equations (a), which contain an arbitrary scalar on the right hand side.)

To establish these relations, it was necessary to evaluate the diagrams on various ``states'', i.e. labelings of the diagrams by vertices and edges from the connection 4-graphs of the appropriate objects. The states were evaluated by decomposing the diagrams into tensor products and compositions of the elementary intertwiners $1 \rightarrow \alpha^2, 1 \rightarrow \kappa \bar{\kappa }, 1 \rightarrow \bar{\kappa} \kappa, \rho \rightarrow \kappa \bar{ \kappa}  $, and $\rho \alpha \kappa \rightarrow \alpha \rho \alpha \kappa $. These elementary intertwiners act on edges by explicit formulas given by gauge transformation matrices.   

As a consequece of the existence of the Q-system, we obtain a subfactor with
index $\frac{7+\sqrt{17}}{2}$ and graphs $$\hpic{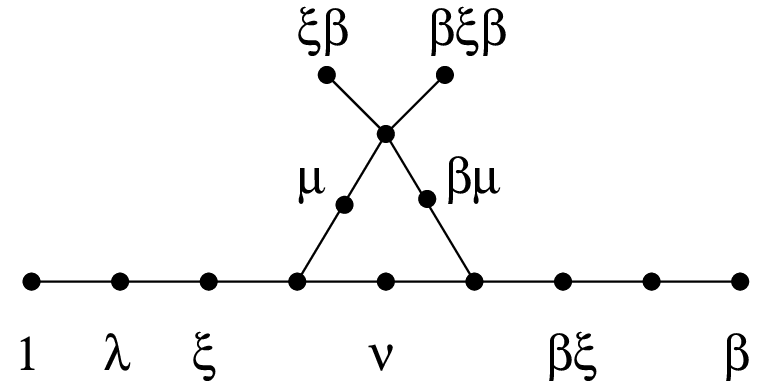} {1 in} \text{ and } \hpic{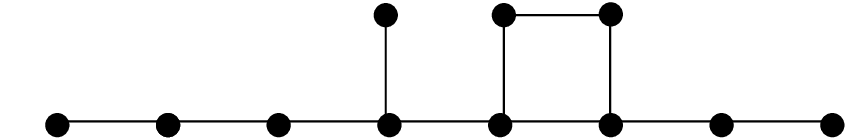} {0.25in} .$$

Again we label the simple objects in the principal even part as well as the generating module object $\lambda $.
\subsection{Existence of AH+2}
Because of the similar fusion structure of this new subfactor, it was conjectured in \cite{MR2812458} that the procedure could be iterated once more to obtain a Q-system for the object $1+ \bar{\lambda} \beta \lambda$ in the dual even part (which is the same category as the dual even part of the original Asaeda-Haagerup subfactor.)

It turns out that this is indeed the case, and we briefly sketch the argument here. The computation consists of two parts: first we must replicate Asaeda and Haagerup's computation of the connections for $\lambda, \xi, \beta$ and the intertwiner $\xi \beta \lambda \rightarrow \beta \xi \beta \lambda$. Then we must verify that the corresponding relations \ref{algebrarelations} above hold for these connections. 

In Figure 2
we indicate what the $4$ graphs are, using a labeling and display similar to that used by \cite{MR1686551}. Note that in the figure we have ``unwrapped the square'', so reading from top to bottom, we have first the upper, then the right, then the lower, then the left graphs. 

\begin{figure}\label{4graphpics}
\centering
\hpic{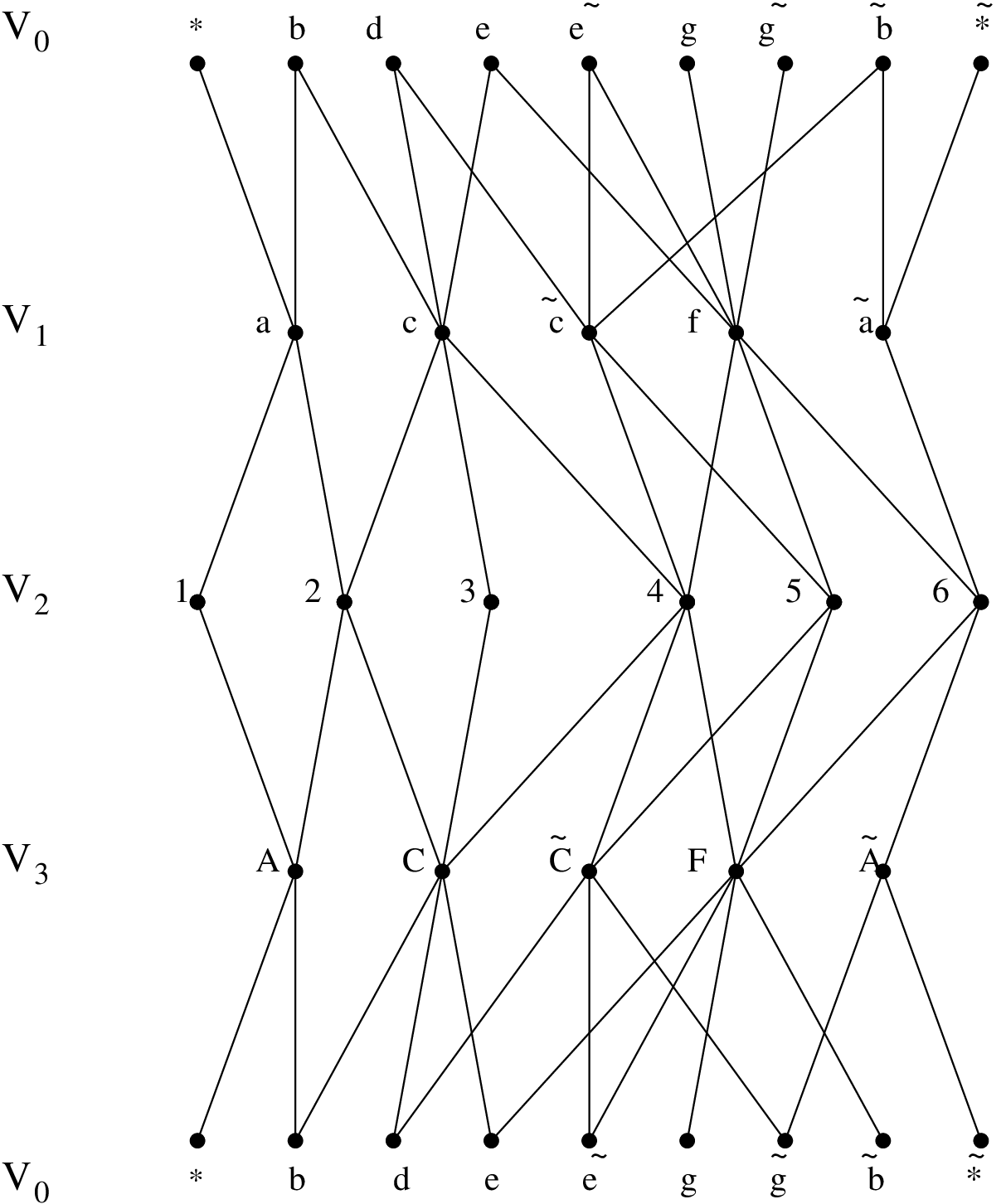} {3 in}, \hpic{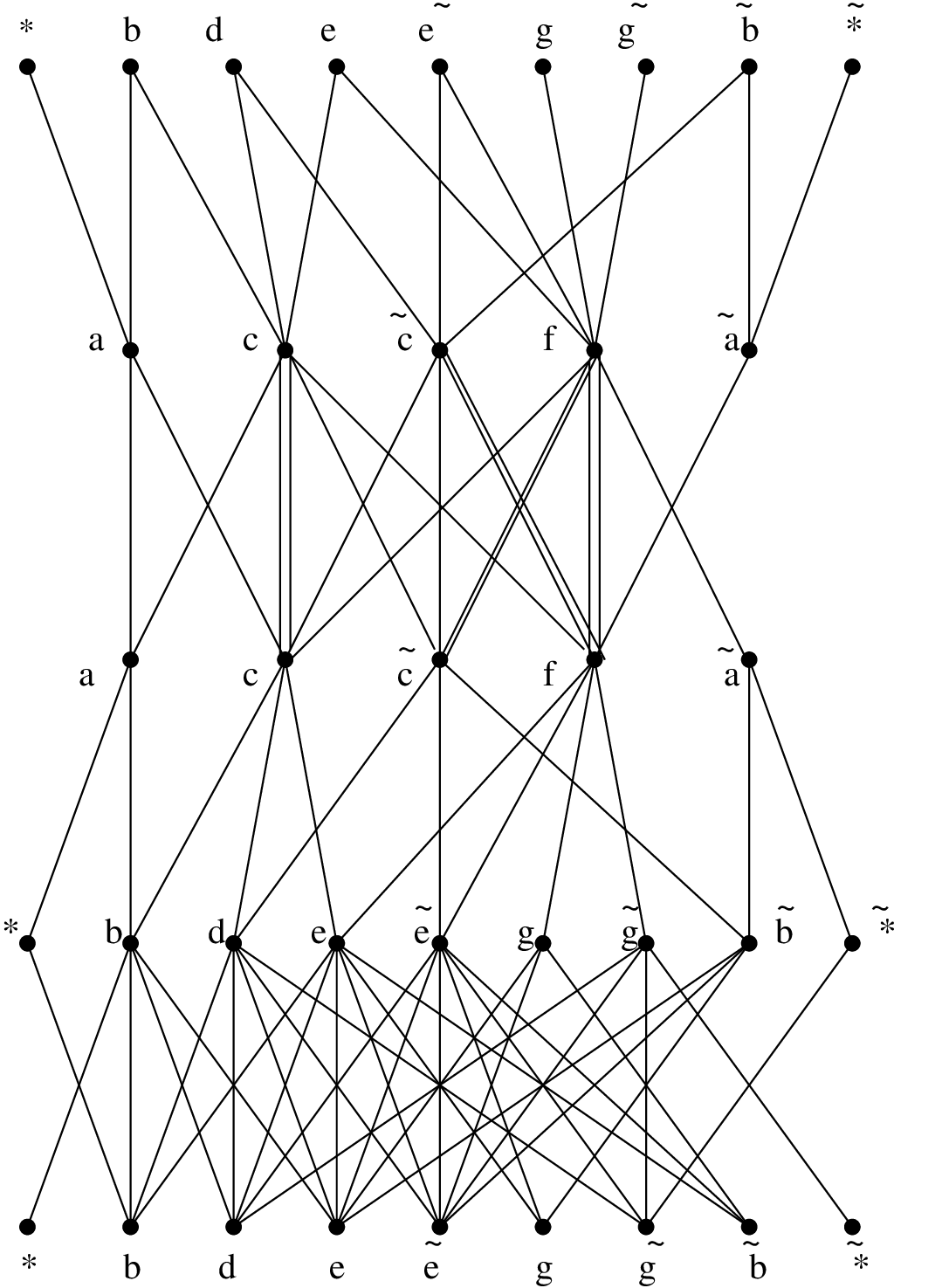} {3in}

\caption{ The $4$-graphs for the connections of $\lambda$ (left) and $\xi$ (right) }

\end{figure}

Note that if $N \subset M $ is the subfactor with Q-system $\lambda \bar{\lambda} $, then the $4$-graphs for the connections of any of the associated $N-M$ (resp. $N-N$) bimodules has the same horizontal graphs (the first and third levels form the top in Figure 2
) as that of  $\lambda$ (resp. $\xi$.) Therefore to specify the $4$-graph of such a bimodule it suffices to specify the vertical graphs (the second and fourth levels in 
Figure 2).

The vertical graphs of the $4$-graph for the identity $N-N$ bimodule have exactly one edge emanating from each vertex in $V_1,V_3 $ connecting to the vertex with the same labeling in $ V_{i+1}$. The vertical graphs of the $4$-graph for the $N-N$ bimodule $\beta$ have exactly one edge emanating from each vertex $x$ in $V_1,V_3 $ connecting to the vertex $\tilde{x} \in V_{i+1}$, where $\tilde{\tilde{x}} = x$ and $\tilde{x}=x $ if there is no vertex labeled $\tilde{x}$.

\begin{lemma}
(a) There is a unique connection on the $4$-graph for $\lambda$ up to gauge choice, which may be taken to be real.\\
(b) There are exactly two real connections on the $4$-graph for the automorphism $\beta$ up to vertical gauge choice - one is identically $1$, the other has the value $-1$ on one cell.\\
(c) The connection for $\beta$ is not identically $1$.  
\end{lemma}

The uniqueness of the connection for $\lambda $ was first pointed out to the authors by Marta Asaeda, who also corrected some signs in the computation.

We now describe a version of the connection for $\lambda $ using the following notation. We refer to Figure 2
for the labelings of vertices. Then the connection is given by matrices corresponding to pairs $u-v$ with $u \in V_0 $ and $v \in V_2 $, where the rows and columns are indexed by $V_3$ and $V_1$, respectively. (You should read $u-v$ as ``between $u$ and $v$'' not ``$u$ minus $v$.'') These matrices are necessarily square whenever they are nonempty. 

In this case the connection consists of $5$ $2 \times 2$ matrices and a bunch of $1 \times 1$ matrices; for the $1 \times 1$ matrics we suppress the matrix notation and simply refer to the entry as $u-v$. Following \cite{MR1686551} we introduce the positive numbers $\beta_n = \sqrt{\displaystyle \frac{7+\sqrt{17}}{2}-n } $, $n \leq 5 $. Then the connection is:

\begin{center}
$\begin{array}{c |c  c}

b-2 & a  & c  \\
  \hline
   A  & \displaystyle \frac{-1}{\beta_1^2} & \displaystyle \frac{\beta \beta_2}{\beta_1^2} \\ C & \displaystyle \frac{\beta \beta_2}{\beta_1^2} & \displaystyle \frac{1}{\beta_1^2}\\
\end{array}$ \hspace{.3in}
$\begin{array}{c |c  c}

d-4 & c  & \tilde{c}  \\
  \hline
 C  & \displaystyle \frac{-1}{\beta_{-1}} &\displaystyle\frac{\beta }{\beta_{-1}} \\

  \tilde{C} & \displaystyle\frac{-\beta}{\beta_{-1}} & \displaystyle\frac{-1}{\beta_{-1}}\\
\end{array}$ \hspace{.3in}
$\begin{array}{c |c  c}

e-4 & c  & f  \\
  \hline
 C  & \displaystyle \frac{-\beta_5}{\beta_1 \beta_3} &\displaystyle\frac{\sqrt{2} \beta}{\beta_1 \beta_3}\\

  F & \displaystyle\frac{\sqrt{2} \beta}{\beta_1 \beta_3} & \displaystyle \frac{\beta_5}{\beta_1 \beta_3}\\
\end{array}$

\vskip2ex

$\begin{array}{c |c  c}

\tilde{e}-4 & \tilde{c}  & f  \\
  \hline
 \tilde{C}  & \displaystyle \frac{2}{\beta_{-1}} &\displaystyle\frac{- \beta_{3}}{\beta_{-1}}\\

  F &\displaystyle\frac{ \beta_{3}}{\beta_{-1}} & \displaystyle \frac{2}{\beta_{-1}}\\
\end{array}$ \hspace{.5in}
$\begin{array}{c |c  c}

\tilde{e}-5 & \tilde{c}  & f  \\
  \hline
 \tilde{C}  & \displaystyle \frac{2}{\beta_1^2} & \displaystyle\frac{\sqrt{2} \beta_{-1}}{\beta_1 \beta_2}\\

  F & \displaystyle\frac{-\sqrt{2} \beta_{-1}}{\beta_1 \beta_2} & \displaystyle \frac{2}{\beta_1^2}\\
\end{array}$
\end{center}

The $1 \times 1$ entries $e-2$, $\tilde{e}-6$, and $g-5$ are $-1$; all the other $1 \times 1$ entries are $1$.

For $\beta$ all the matrices are $1 \times 1$. We take all the entries to be $1$ except for $e-f$, which we take to be $-1$. That $\beta $ cannot be identically $1$ (part (c) of the lemma above) was discovered by trial and error.  Then we compute a version of the connection for $\xi$, which is uniquely determined up to vertical gauge choice. 

With this information, we can compute all the necessary elementary intertwiners. The computations are however lengthy.  The similar calculation in \cite{MR1686551} takes almost $30$ pages, and this calculation is somewhat more involved. We used Mathematica for bookkeeping and to multiply matrices of algebraic numbers, using its \textit{RootReduce} function.
 The results are described in the accompanying note \textit{AHplus2.pdf} which may be found in the arXiv source for this paper, along with a Mathematica notebook containing the main calculation.

\begin{lemma}
 The relations \ref{algebrarelations} hold when $\rho, \alpha, \kappa$ are replaced by $\xi, \beta, \lambda$ respectively.
\end{lemma}
\begin{proof}
Again, by direct computation, following the methods of \cite{MR2812458}. The details are contained in \textit{AHplus2.pdf}.
\end{proof}

\begin{theorem}
 The object $1 + \bar{\lambda } \beta \lambda$ admits a Q-system.
\end{theorem}
\begin{proof}
 As in \cite{MR2812458} we may choose $c$ in relation (a) to be $\sqrt{d(\xi)}$, and the proof proceeds exactly as in the main theorem there.
\end{proof}

By the above result we now have a third subfactor with index $\frac{9+\sqrt{17}}{2} $ and graphs $$\hpic{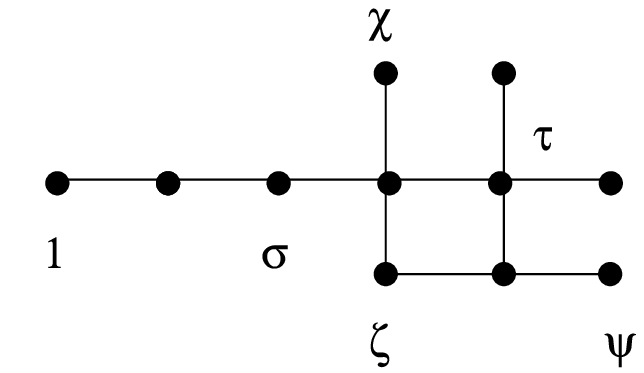} {1in} \text{ and } \hpic{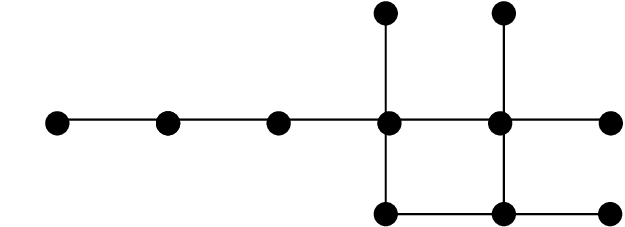} {0.6in} .$$ 

\subsection{Relations among the Asaeda-Haagerup fusion categories}

The three subfactors discussed in the preceeding sections all have the same dual even part but different principal even parts. That the principal even parts are different can be seen immediately by checking the Frobenius-Perron weights of the even vertices of the principal graphs, which are different for the three different graphs. In fact, the complete fusion rules for each of the three principal even parts may be deduced from the corresponding graph. The dual data is all trivial except that $\rho \alpha $ is dual to $\alpha \rho $ and similarly $\xi \beta $ is dual to $\beta \xi$. We include the multiplication tables below.

\begin{table}
\begin{tabular}{ c || c | c | c | c | c }

                 & $\psi$             & $\chi$       & $\sigma $ & $\zeta$          & $\tau$     \\ \hline \hline
$\psi$              & $1 + \psi + \zeta $             & $\chi + \tau $       & $\zeta + \tau $ & $\psi + \sigma + \zeta + \tau$          & $\chi + \sigma + \zeta + 2\tau$     \\ \hline
$\chi$         &      & $1+ \psi+\chi+ \tau$     & $\sigma + \zeta + \tau $        & $\sigma + \zeta + 2\tau$   & $ \Lambda +\zeta + \tau$  \\ \hline
$\sigma $      &       &             & $1 + \chi + \sigma + \zeta + \tau $   & $\Lambda +\tau$ & $\Lambda + \zeta + 2\tau$           \\ \hline
$\zeta$          &            &     &         & $1+\Lambda + \zeta + 2\tau$ & $\Lambda+ \chi + \sigma + 2\zeta+3\tau$ \\ \hline
$\tau $    &    &           & & & $1+2\Lambda+ \sigma+2\zeta+4\tau$ \\ 
\end{tabular}

\medskip

 $\mathscr{AH}_1 $ has $6$ simple objects, which are ordered in the associated data files $1, \chi, \psi,\tau, \sigma,\zeta $, and have dimensions $1,\frac{d+1}{2},\frac{d-1}{2},\frac{3d-1}{2},d,\frac{d+3}{2} $, respectively, where $d=4+\sqrt{17}$. 

\medskip
 
 We use the abbreviation $\Lambda=\psi + \chi + \sigma + \zeta + \tau $. Since the multiplication is commutative, we omit the sub-diagonal entries.

\caption{ $AH_1$ multiplication table.}
\end{table}

\begin{table}

\begin{tabular}{ c || c | c | c}

                 & $\rho$             & $\pi$       & $\eta $  \\ \hline \hline
$\rho$          & $1 + \rho + \pi$ & $\rho + \Gamma$       & $\alpha \rho + \alpha \rho \alpha + \Gamma$  \\ \hline
$\pi$         & $\rho + \Gamma$        & $1 + \Delta + 2 \Gamma$     & $\Delta + 2 \Gamma + \eta $        \\ \hline
$\eta$      & $\rho \alpha + \alpha \rho \alpha + \Gamma$      & $\Delta+2\Gamma+ \eta$            & $1+\alpha+\Delta+2\Gamma+\pi+\alpha\pi $  \\ 
\end{tabular}
\medskip

 $\mathscr{AH}_2 $ has $9$ simple objects, which are ordered in the associated data files $1, \alpha, \rho, \alpha \rho, \rho \alpha, \alpha \rho \alpha, \pi, \alpha \pi, \eta $, and have dimensions $1,1,\frac{d-1}{2},\frac{d-1}{2},\frac{d-1}{2},\frac{d-1}{2},d,d,d+1 $, respectively, where $d=4+\sqrt{17}$. 

\medskip

Rules involving $\alpha$: $\alpha^2=1$, $\pi \alpha=\alpha \pi$, $\alpha \eta = \eta\alpha=\eta$, $\rho \alpha \rho = \alpha \rho \alpha + \eta $.\\
 We use the abbreviations $\Gamma = \pi + \alpha \pi + \eta$, $\Delta= \rho + \alpha \rho + \rho \alpha + \alpha \rho \alpha $.

\caption{ $AH_2$ partial multiplication table.} 
\end{table}

\begin{table}

\begin{tabular}{ c || c | c | c}

                 & $\xi$             & $\mu$       & $\nu $  \\ \hline \hline
$\xi$          & $1 + \xi + \mu+ \nu$ & $\xi + \beta \xi + \beta \xi \beta + \Pi$       & $\xi + \xi \beta + \Pi$  \\ \hline
$\mu$         & $\xi + \xi \beta + \beta \xi \beta + \Pi $        & $1 + \Pi + 2 \Psi $     & $\Pi + \Psi + \mu + \beta\mu $        \\ \hline
$\nu$      & $\xi + \beta \xi + \Pi$      & $\Pi + \Psi + \mu + \beta \mu$            & $1+\beta+ \Pi + \Psi + \nu$  \\ 
\end{tabular}

\medskip

 $\mathscr{AH}_3 $ has $9$ simple objects, which are ordered in the associated data files $1, \beta, \xi, \beta \xi, \xi \beta, \beta \xi \beta, \mu, \beta \mu, \nu $, and have dimensions $1,1,\frac{d+1}{2},\frac{d+1}{2},\frac{d+1}{2},\frac{d+1}{2},d,d,d-1 $, respectively, where $d=4+\sqrt{17}$. 

\medskip

Rules involving $\beta$: $\beta^2=1$, $\mu \beta=\beta \mu$, $\beta \nu = \nu\beta=\nu$, $\xi \beta \xi = \beta \xi \beta + \mu +\beta \mu$.\\
 We use the abbreviation $\Pi = \mu + \beta \mu + \nu, \Psi= \xi + \beta \xi + \xi \beta + \beta \xi \beta $.
\caption{ $AH3$ partial multiplication table.} 
\end{table}

We will call the three subfactors AH, AH+1, and AH+2. We will call the fusion category which is the dual even part of all three subfactors $\mathscr{AH}$, the principal even part of AH+2  $\mathscr{AH}_1$, the principal even part of AH  $\mathscr{AH}_2$, and the principal even part of AH+1  $\mathscr{AH}_3$. The Grothedieck rings of these fusion categories will be called $AH_1$, $AH_2$, and $AH_3$ respectively.

\begin{lemma} \label{ah_object}
 Let $\mathscr{C} $ be a unitary fusion category containing an object $\xi$ such that the Frobenius-Perron dimension $dim(\xi)=\frac{3+\sqrt{17}}{2}$ and such that $\xi^2 \cong 1 + \xi + \eta $ where $\eta $ is a simple object. Then $\mathscr{C} $ is equivalent to either $\mathscr{AH}_1$ or $\mathscr{AH}_2$.
\end{lemma}
\begin{proof}
 This follows immediately from Theorem \ref{4stqs} and the uniqueness (up to duality) of the finite-depth subfactor with index $\frac{5+\sqrt{17}}{2}$.
\end{proof}

\begin{corollary}
 We have $\mathscr{AH} \cong \mathscr{AH}_1$.  Thus, AH+2 is an autoequivalence.
\end{corollary}
\begin{proof}
 The object $\psi$ and the corresponding object on the dual graph both satisfy the conditions of the previous lemma, and clearly neither even part is $\mathscr{AH}_2 $.
\end{proof}

Note that tensoring with $\alpha$ fixes the middle vertex of the $AH$ graph, and hence the tensor subcategory generated by $\alpha$ has $\mathrm{Vec}$ as a module category and thus has trivial associator.  As a consequence, $1+\alpha$ has a unique algebra structure given by the group ring of $\mathbb{Z}/2\mathbb{Z}$.

\begin{theorem} \label{index2}
 The category of bimodules over the two-dimensional algebra object $1+ \alpha$ in $\mathscr{AH}_2$ is equivalent to $\mathscr{AH}_3$. Similarly, $\mathscr{AH}_2$ is equivalent to the category of $(1+\beta)-(1+\beta)$ bimodules in  $\mathscr{AH}_3$. 
\end{theorem}

\begin{proof}
 Let $\lambda$ be an $\mathscr{AH}_3-\mathscr{AH}_1$ bimodule corresponding to the Asaeda-Haagerup subfactor, as above. Let $\mu$ be an object in an invertible $\mathscr{C}-\mathscr{AH}_3$ bimodule category such that $\bar{\mu} \mu = 1 + \beta$. Then since $(1+ \beta, 1 + \xi)=1 $, by Frobenius reciprocity $\mu \lambda $ is an irreducible $\mathscr{C}-\mathscr{AH}_3$ bimodule. We can compute (see the discussion of fusion computations in Section 5) that the dual graph of $\mu \lambda$ must be $$ \vpic{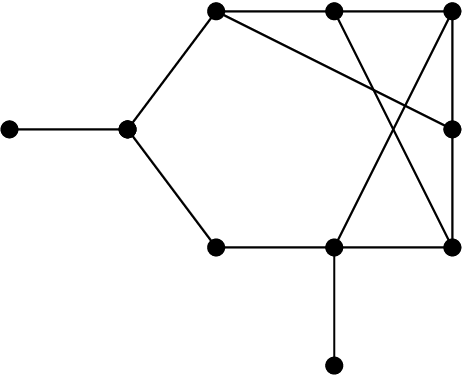} {1in} $$ and then that the principal graph must be $$ \vpic{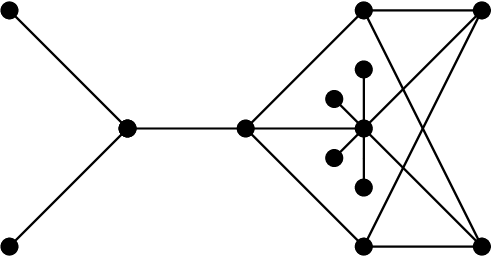} {1in} $$

This implies that the simple objects of $\mathscr{C} $ have the same Frobenius-Perron dimension as those of  $\mathscr{AH}_2$. We can compute all consistent fusion rules for objects of those weights. There are several solutions, but every solution has at least one object satisfying the conditions of Lemma \ref{ah_object}. Therefore $\mathscr{C} \cong \mathscr{AH}_2$, and the conclusion follows.
\end{proof}

\begin{theorem}\label{no_out}
 The fusion categories $\mathscr{AH}_1$ and $\mathscr{AH}_2$ do not admit any outer automorphisms.
\end{theorem}

\begin{proof}
 Consider the algebra objects corresponding to the Asaeda-Haagerup subfactor in $\mathscr{AH}_1$ and $\mathscr{AH}_2$.  These algebra objects are all unique up to inner automorphism, so if any of these categories admit a non-trivial outer automorphism, we can find an outer automorphism which leaves one of these algebras invariant. Since each such algebra object admits a unique algbera structure by $3$-supertransitivity, the outer automorphism must act trivially on the algebra.  Hence it is enough to show that the Asaeda-Haagerup subfactor planar algebra does not admit any outer automorphism.  Since there are no other subfactors of index $\frac{5+\sqrt{17}}{2}$, this planar algebra is generated by a single $6$-box which is a rotational eigenvector and satisfies a certain quadratic equation (see \cite{MR2972458}).  Any outer automorphism would send this rotational eigenvector to a multiple of itself, and the only multiple which satisfies the quadratic equation is itself.  Hence there are no outer automorphisms.
\end{proof}

As we will see later $\mathscr{AH}_3$ also has no outer automorphisms.  This is not terribly difficult to show directly, but such a direct argument is more tedious since we can't use uniqueness of a subfactor of that specific index.

The three categories $\mathscr{AH}_1, \mathscr{AH}_2, \mathscr{AH}_3$ will be called the Asaeda-Haagerup fusion categories.

\section{The combinatorics of Brauer-Picard groupoids}

We first summarize the key ideas of the next two sections, and then give rigorous definitions and statements.

Any time you have a category, you can \emph{decategorify} by turning isomorphisms into equalities (i.e. taking the Grothendieck group).  For example, fusion categories decategorify to give fusion rings.  In general you lose a lot of information passing from a category to its decategorification.  For example, different fusion categories can give the same fusion ring, and some fusion rings come from no fusion categories.  However, in many cases you can prove results only by looking at the combinatorics of fusion rings (see \cite{MR1145672, MR1424954} for some examples).  Decategorifying the whole Brauer-Picard groupoid yields a bunch of structures.  Each fusion category gives a fusion ring, each bimodule category yields a fusion bimodule, and the composition gives a composition rule for fusion bimodules.  These structures satisfy a bunch of compatibility conditions.

Perhaps the most interesting part of this structure is the composition rules.  Given two bimodule categories ${}_\mathscr{C} \mathscr{M}_\mathscr{D}$ and ${}_\mathscr{D} \mathscr{N}_{\mathscr{E}}$ we can form the relative tensor product ${}_\mathscr{C} \mathscr{M} \boxtimes_\mathscr{D} \mathscr{N}_{\mathscr{E}}$.  In particular, given objects $m \in \mathscr{M}$ and $n \in \mathscr{N}$, we get an object $mn$ in the tensor product.   Taking Grothendieck groups, we have two bimodules ${}_{K(\mathscr{C})} K(\mathscr{M})_{K(\mathscr{D})}$ and ${}_{K(\mathscr{D})} K(\mathscr{N})_{K(\mathscr{C})}$, but their tensor product is typically not their tensor product as bimodules.  In fact, knowing the bimodule structures on $K(\mathscr{M})$ and $K(\mathscr{N})$ is not enough to determine $K(\mathscr{M} \boxtimes_{\mathscr{D}} \mathscr{N})$ even as an abelian group!  Instead the categorical structure yields extra information: a composition map $K(\mathscr{M}) \otimes_{K(\mathscr{D})} K(\mathscr{N}) \rightarrow K(\mathscr{M} \boxtimes_{\mathscr{D}} \mathscr{N})$.  Perhaps surprisingly, in our case there are usually not very many possible compositions.

By considering the whole decategorified Brauer-Picard groupoid at once, we can often rule out possible fusion rings or fusion bimodules which look fine locally.  For example, suppose we have a candidate fusion bimodule $M$ over two fusion rings $A$ and $B$, and further suppose we have a known bimodule category which decategorifies to give a fusion bimodule ${}_B N_C$.  If $M$ can be realized by an actual bimodule category, then it follows that there must be some composition rule sending $M \otimes_B N$ to a valid bimodule between $A$ and $C$.  If there are no such composition rules then we can rule out $M$ as coming from an actual bimodule category.

\subsection{Fusion modules, fusion bimodules, and multiplication maps}
Any semisimple module category over a fusion category induces a representation of the Grothendieck ring of the fusion category on the Grothendieck group of the module category. We first recall the Frobenius-Perron theory of fusion rings and modules. For more details, see \cite{MR2183279}.

\begin{definition}
A \textit{fusion ring} $(F,S)$ is a ring $F$ whose additive group is the free Abelian group on a 
finite set $S$ containing $1$, which is endowed with an involution denoted by $\xi \mapsto \bar{\xi}, \xi \in S$ which extends to an anti-involution on $F$.   In addition, these structures should satisfy the following condition,
 $\xi \eta = \sum_{\mu \in S} \limits N^{\xi}_{\eta \mu}$ for 
all $\xi, \eta \in S$, where the $N^{\xi}_{\eta \mu}$ are non-negative 
integers such that $N^{\xi}_{\eta 1} = \delta_{\xi, \bar{\eta}}$. 
The Frobenius-Perron dimension $d(\xi)$ of an element $\xi in F$ is its image under the unique nonzero homomorphism $F \rightarrow \mathbb{R}$ which maps $S$ into the set of positive numbers. 
\end{definition}

Here the basis $S$ and the involution are considered part of the data of the fusion ring. The Grothendieck ring of a fusion category is always a fusion ring, but there are many fusion rings which do not arise as the Grothendieck ring of any fusion category. We are interested in representations of fusion rings which respect the fusion structure.

\begin{definition}
 A (left) fusion module $(M,T) $ over a fusion ring $(F,S)$ is a finite set $T$ along with an indecomposable (left) representation of the fusion ring as endomorphisms of the free Abelian group on $T$ such that the action satisfies $\xi \eta = \sum_{\mu \in T} \limits N^{\xi}_{\eta \mu}$ for all $\xi \in S, \eta, \mu \in T$, where the $N^{\xi}_{ \eta \mu}$ are non-negative integers such that $N^{\xi}_{\eta \mu} = N^{\bar{\xi}}_{\mu \eta}$ for all $\xi, \mu, \eta $. A Frobenius-Perron dimension vector is a positive real vector indexed by $T$ which is an eigenvector for all the matrices $N^{\xi}, \xi \in S$. Such a dimension vector exists and is unique up to scalar multiples. \end{definition}

In a similar way, we can define right fusion modules. There are obvious notions of isomorphisms of fusion rings and fusion modules (namely there should be bijections on the basis sets which preserve the algebraic structure). 

\begin{remark}
Fusion rings are also known as unital based rings of finite rank. Fusion modules are sometimes called \textit{nimreps}, which stands for non-negative integer matrix representations.
\end{remark}

Let $(F,S) $ be a fusion ring with $S=\{\xi_1,...,\xi_m \} $. The right and left fusion matrices $R^{\xi} $ and $L^{\xi} $ of an element $\xi \in F $
are given by $(R^{\xi})_{\mu \eta}= (\eta \xi, \mu)$ and
$(L^{\xi})_{\mu \eta}= ( \xi \eta, \mu)$, $\eta, \mu \in S $ . If $(M,T) $ is a left fusion module over $(F,S) $, the fusion matrix  $A^{\kappa} $  of an element $\kappa \in M $ is given by 
$(A^{\kappa})_{\lambda \xi}=(\xi \kappa, \lambda) $, $\xi \in S, \ \mu \in T $. Fusion matrices for right fusion modules are defined similarly.

Given a fusion module $(M,T)$ over $(F,S)$, we denote by $(\mu,\eta)$ the dot product of two elements of $M$ with respect to the basis $T$, and similarly for two elements of $F$ (with respect to the basis $S$.) We also define \textit{right multiplication by duals} for a left fusion module as follows: $ \mu \bar{\lambda} := \sum_{\xi \in S} \limits (\mu, \xi \lambda) \xi  $ for $\mu, \lambda \in M $. Note that $\bar{\lambda} $ is not actually an element of the fusion module; the expression is just a formal argument for multiplication. Similarly for a right fusion module we may define left multiplication by duals.

\begin{lemma}
Let $(M,T) $ be a left fusion module over $(F,S) $. Then right multiplication by duals 
$ \mu \bar{\lambda}$ is biadditive in $\mu $ and $\lambda $,
we have $\overline{\mu \bar{\lambda}}=\lambda \bar{\mu} $,
and
$$(\xi \mu) \bar{\lambda}=\xi (\mu \bar{\lambda}), \ \forall \mu, \lambda \in T, \ \xi \in F. $$
%
\end{lemma}
\begin{proof}
We prove the last assertion. For any $\eta \in S $, we have
$$(\xi (\mu \bar{\lambda}), \eta ) =(\mu \bar{\lambda}, \bar{\xi} \eta ) =
\sum_{\zeta \in S} \limits (\zeta \lambda, \mu  ) 
(\zeta,\bar{\xi} \eta)=(\mu, \bar{\xi} \eta \lambda )=(\xi \mu, \eta \lambda)=((\xi \mu) \bar{\lambda}, \eta) .$$

%
\end{proof}
A similar associativity holds for left multiplication by duals in right fusion modules. We recall the following easy calculation.
\begin{lemma} \label{pglem}
 Let $(M,T) $ be an fusion module over a fusion ring $(F,S) $. For $\kappa, \lambda \in T $,
 $R^{\kappa \bar{\lambda}}=(A^{\lambda})^T A^{\kappa} $.
\end{lemma}

The canonical Frobenius-Perron dimension (for module categories) is mentioned in \cite{MR2677836}; we include a proof of existence for the convenience of the reader.
\begin{lemma}
Let $(M,T) $ be a left fusion module over $(F,S) $. There is a normalization $d_T$ of the Frobenius-Perron dimension vector of $T$ such that
$$d(\mu \bar{\lambda})=d_T(\mu)d_T(\lambda), \ \forall \ \mu,\lambda \in T .$$ 
\end{lemma}

\begin{proof}

Let $ d_S$ be the Frobenius-Perron dimension vector for $S$.
For each $\mu \in T $, set $d_T^{\mu}=A^{\mu} d_S$; this define an additive dimension function on $M$ (we use the notation $d_T^{\mu}$ both for the vector and for the corresponding dimension function). Then for any $\lambda \in T $,  we have
$$(A^{\lambda})^T d_T^{\mu} =(A^{\lambda})^TA^{\mu} d_S=R^{\mu \bar{\lambda}} d_S=d(\mu \bar{\lambda})d_S,$$ and then for any $\xi \in S $ we have $$d_T^{\mu}(\xi \lambda)=A^{\lambda}_{\cdot \xi} \cdot d_T=((A^{\lambda})^Td_T^{\mu})_{\xi}=  d(\mu \bar{\lambda})d(\xi)=d_T^{\mu}(\lambda)d(\xi).$$ This shows that $d_T^{\mu} $ is a Frobenius-Perron dimension function for each $\mu $, and in particular
we may choose positive scalars $a_{\mu} $ such that $a_{\mu} d_T^{\mu}=a_{\lambda}d_T^{\lambda} ,\ \forall \mu, \lambda \in T$. Fix $\mu, 
\lambda \in T  $. Then we have 
\begin{multline*}
d(\mu \bar{\lambda})d_S=(A^{\lambda})^TA^{\mu}d_S=(A^{\lambda})^T d_T^{\mu}=\frac{a_{\lambda}}{a_{\mu}} 
(A^{\lambda})^T d_T^{\lambda}=\frac{a_{\lambda}}{a_{\mu}}(A^{\lambda})^TA^{\lambda} d_S=
\frac{a_{\lambda}}{a_{\mu}} d(\lambda \bar{\lambda})d_S
\end{multline*}
so that $a_{\mu}d(\mu \bar{\lambda})=a_{\lambda} d(\lambda \bar{\lambda})$. Similarly,
we have $a_{\lambda}d(\lambda \bar{\mu})=a_{\mu} d(\mu \bar{\mu})$. Since $d(\mu \bar{\lambda})= d(\lambda \bar{\mu})$, this gives $(a_{\lambda})^2 d(\lambda \bar{\lambda})=(a_{\mu})^2d(\mu \bar{\mu})$. Since this holds for any $\mu, \lambda \in T $,
we have $ a_{\mu}=\frac{c} { \sqrt{d(\mu \bar{\mu})}}$ for some constant $c$ and all $\mu \in T $. Taking $c=1 $ and setting $d_T=a_{\mu} d_T^{\mu} $ for some $\mu $ gives the desired normalized dimension vector.

%
%
%

\end{proof}
%
%
%
%
%

Note that any basis elements $\xi \in S $ and $\mu \in T$, we have $$d(\xi)d(\mu) = d(\xi \mu ) \geq
d(\mu)(\xi \mu, \mu)=d(\mu)(\mu \bar{\mu},\xi) ,$$ so we get the bound
$(\mu \bar{\mu},\xi) \leq d(\xi) $. This will be important when searching for
 division algebras in fusion categories.

\begin{definition}
A fusion bimodule $(M,T) $ is a free Abelian group on the basis $T$
along with a left fusion module structure over a fusion ring $(F,S) $ and a right fusion module structure over a fusion ring $(G, R) $,
such that $$\xi (\mu \eta)= (\xi \mu )\eta, \ \forall \ \xi \in S, \ \mu \in T, \ \eta \in R  $$
and
$$  \mu (\bar{\lambda} \kappa)=  (\mu \bar{\lambda}) \kappa, \ \forall \ \mu, \lambda, \kappa \in T .$$
\end{definition}

\begin{lemma} \label{lem:leftright}
Let $(M,T) $ be a fusion bimodule. Then the left canonical Frobenius-Perron dimension coincides
with the right canonical Frobenius-Perron dimension.
\end{lemma}
\begin{proof}
For any $\kappa, \lambda, \mu \in T$, the relation $\kappa ( \bar{\lambda} \mu) =(\kappa \bar{\lambda}) \mu $ means that the column $v^{\kappa}_{\lambda, \mu} $ corresponding to $\kappa $ in the matrix (with respect to $T$) of right multiplication by $\bar{\lambda }\mu $ is the same as  the column $w^{\mu}_{\kappa, \lambda} $ corresponding to $\mu $ in the matrix (again with respect to $T$) of left multiplication by $\kappa \bar{\lambda}  $. Let $d_L $ be the left canonical Frobenius-Perron dimension of $M$. Then $v^{\kappa}_{\lambda, \mu} \cdot d_L=w^{\mu}_{\kappa, \lambda}\cdot d_L =d_L(\mu)d_L(\kappa)d_L(\lambda) $. This shows that that $d_L$ is an eigenvector of right multiplication by $\bar{\lambda }\mu $ with eigenvalue $ d_L (\lambda)d_L(\mu)$ for all $\lambda, \mu \in T $. Therefore $d_L $ agrees with the right canonical Frobenius-Perron dimension up to a scalar, and since we have for each $\lambda,\mu  \in T$ that the Frobenius-Perron eigenvalue of right multiplication by $\bar{\lambda }\mu $ is $d(\bar{\lambda }\mu ) =d_L (\lambda)d_L(\mu)$, that scalar is $1$.

\end{proof}

\begin{lemma}
The Grothendieck group of an indecomposable module category over a fusion category has the structure of a
fusion module over the Grothendieck ring of the fusion category. The Grothendieck group of 
an invertible bimodule category over a pair of fusion categories has the structure of a fusion bimodule over the Grothendieck rings
of the corresponding fusion categories.
\end{lemma}
\begin{proof}
The only subtle part is to verify
$ \mu (\bar{\lambda} \kappa) \cong  (\mu \bar{\lambda}) \kappa$ for objects in a bimodule category,
which follows from the associativity of left and right internal homs (Lemma \ref{basid}).
\end{proof}

Let $\mathscr{A}$, $\mathscr{B}$, and $\mathscr{C} $ be fusion categories with fusion rings $A$, $B$, and $C$, respectively Let ${}_{\mathscr{A}} \mathscr{K} { }_{\mathscr{B}}$, ${}_{\mathscr{B}} \mathscr{L} { }_{\mathscr{C}}$, and $ {}_{\mathscr{A}} \mathscr{M} { }_{\mathscr{C}}  $ be invertible bimodule categories with fusion bimodules ${}_A K {}_B$, $ {}_B L {}_C$, and $ {}_A M {}_C $, respectively, such that $\mathscr{K} \boxtimes_{\mathscr{B}} \mathscr{L} \cong \mathscr{M} $.  The equivalence of bimodule categories  $\mathscr{K} \boxtimes_{\mathscr{B}} \mathscr{L} \cong \mathscr{M} $ induces an $A-C$ bimodule map from $K \otimes_B L $ to $M$ 
which takes tensor products of basis elements of $ K$ and $L$ to non-negative combinations 
of basis elements in $M$. 

%

The associativity of internal homs (Lemma \ref{inthomass}) implies that this map preserves dimension and multiplication by duals in the following sense: if $ \xi_1 \otimes \eta_1  \mapsto \sum a_i \mu_i$ and  $ \xi_2 \otimes \eta_2  \mapsto \sum b_i \mu_i$ where $\xi_1, \xi_2 \in K$, $\eta_1, \eta_2 \in L$ are basis elements and the sums are over the basis $\{\mu_i \}$ of $ M$, then for any $\lambda \in A $ we have $$( \xi_1 (\eta_1 \bar{\eta}_2) \bar{\xi}_2, \lambda)=\sum a_i b_j (\mu_i \bar{\mu}_j, \lambda ) $$ and for any $\sigma \in C$ we have $$( \bar{\xi}_1 (\bar{\eta}_1 \eta_2) \xi_2, \sigma)=\sum a_i b_j (\bar{\mu}_i\mu_j, \sigma ) .$$ Similarly, by the last statement in
Lemma \ref{inthomass} we have $$\sum  a_ib_i= (\bar{\xi}_1 \xi_2,  \eta_2 \bar{\eta}_1).$$
We also have $$d(\xi_1) d(\eta_1)=\sum a_i d(\mu_i),$$ where all the dimensions are the canonical Frobenius-Perron dimensions for fusion bimodules.

%
%
%

\begin{definition}\\
 (a) A \textit{multiplication map} on a triple of fusion bimodules $({}_A K {}_B, {}_B L {}_C, {}_A M {}_C) $ is a homomorphism from ${}_A K \otimes_B L {}_C $ to ${}_A M {}_C$ which takes tensor products of basis elements in $K$ and $L$ to non-negative combinations of basis elements of $M$ and preserves dimension and multiplication by duals in the sense of the preceding paragraph. The triple is $(K,L,M)$ said to be \textit{multiplicatively compatible} if there exists such a multiplication map.\\
(b) Similarly, a multiplication map on a triple of fusion modules/bimodules $(K {}_A, {}_A L {}_B, M {}_B) $ is a homomorphism from $ K \otimes_A L {}_B $ to $ M {}_B$ which takes tensor products of basis elements in $K$ and $L$ to non-negative combinations of basis elements of $M$ and preserves dimension and right multiplication by duals in the sense of the preceding paragraph. The triple is $(K,L,M)$ said to be \textit{multiplicatively compatible} if there exists such a multiplication map.\\\end{definition}

Note that in general such a multiplication map, if it exists, is not uniquely determined by the fusion rules.

\subsection{Algorithms for computation}

In this section we give some algorithms for finding the fusion modules for a fusion ring, the fusion bimodules for two fusion rings, and the 
multiplication maps for triples of fusion modules or bimodules.  It is easy to see that these are all finite problems. 
Indeed we are always looking for a finite list of natural numbers, and it is not difficult 
to find bounds for these numbers via dimension considerations.  Thus in principle one could simply enumerate all possibilities and check which work. In practice this is a hopelessly long computation, even for very simple fusion rings, so we give some more efficient algorithms.  These algorithms are still only practical for small fusion rings but are sufficient for the Asaeda-Haagerup rings.

First, we describe a simple algorithm for finding all the decompositions of a positive semi-definite square non-negative integer matrix $M$ into $A^TA$ for non-negative integer matrix $A$ with no zero columns (up to permutations of columns of $A$).

\begin{definition}
 The set of sum of squares decompositions of a positive integer $N$ is the set of all vectors of pairs of positive integers $(a_i,b_i), 1\leq i \leq r$ such that $a_i < a_j $ whenever $i<j$ and $\sum_{i=1}^r \limits b_i a_i^2 =N$.
\end{definition}

\begin{definition}
 An $m$-partial decomposition of a positive semi-definite $n \times n$ non-negative integer matrix $M$ is an $m \times k$ non-negative integer matrix $P$, $0 \leq m \leq n$, $0 \leq k$, such that $P$ has no zero columns and $P_i  \cdot P_j = M_{ij} $ for all $1 \leq i,j \leq m$. We define $P_i \cdot P_j = 0$ if $k=0$.   
\end{definition}
 
Thus the problem is to find all $n$-partial decompositions of $M$. The algorithm proceeds by describing all $(m+1)$-partial decompositions whose first $m$ rows form a given $m$-partial decomposition.

\begin{algorithm} To find all decompositions of a positive semi-definite $n \times n$ non-negative integer matrix $M$ into $A^TA$ for a non-negative integer matrix $A$ with no zero columns (up to permutations of columns of $A$):\\

1) Find the sets of sum of squares decompositions of all diagonal entries of $M$.\\

2) Start with the unique $0$-partial decomposition of $M$ (the empty matrix). Then for every $m$-partial decomposition $P$ of $M$, $m<n$, consider all $(m+1)$-partial decompositions as follows:\\

\indent Case 1: $P$ is empty, $M_{(m+1)(m+1)}=0 $. By positive semi-definiteness of $M$, the unique $(m+1)$-partial decomposition is empty as well.\\

\indent Case 2: $P$ is empty, $M_{(m+1)(m+1)}= N > 0$. Then for each sum of squares decomposition of $N$, $(a_i,b_i), 1 \leq i \leq r$ we get an $(m+1)$-partial decomposition by taking a vector with $b_i$ copies of $a_i$ for each $i$, and then adding $m$ rows of zeros above it to complete the matrix with $m+1$ rows.\\

\indent Case 3: $P$ has $k$ columns, $k > 0$,  $M_{(m+1)(m+1)}=0 $. Then adding a row of zeros to $P$ gives the unique $m+1$-partial decomposition which extends $P$.\\

\indent Case 4: $P$ has $k$ columns, $k > 0$,  $M_{(m+1)(m+1)}= N > 0$. Then $(m+1)$-partial decompositions which extend $P$ correspond to the following data: a sum of squares decomposition of $N$, $(a_i,b_i), 1 \leq i \leq r$, and a vector $v=(v_1,\ldots,v_k ) $ such that:\\ 
(a)$v_j \in \{ a_i | 1 \leq i \leq r\} \cup \{0\} $  for all $1 \leq j \leq k $ \\
(b) $|\{ j |v_j =a_i  \} | \leq b_i$  for all $1 \leq i \leq r $\\
(c) $v \cdot P_l=M_{l(m+1)} $ for all $1 \leq l \leq m $.\\
Given such data, we can form an $(m+1)$-partial decomposition by adding a row to $P$ with values equal to $v$, and then adding a column for each ``leftover'' member of the sum of squares decomposition with zeros above.\\

\end{algorithm}

We can use the above decomposition algorithm to find all possible principal graphs of 
simple objects in simple module categories over a given fusion category following the technique from \cite[\S 3]{1102.2631}.
In the following algorithms we fix for any fusion ring $(F,S) $ an ordering $S=\{\xi_1,...,\xi_m \} $.
Finding left fusion modules and finding right fusion modules are equivalent problems. We write the algorithms for right modules
since those are the ones we actually found with the computer.

\begin{algorithm}
To find all possible fusion matrices of basis elements in fusion modules over a given fusion ring $(F,S) $:\\
1) Compute the left fusion matrix $L^{\xi} $ for each $\xi \in F$ such that $(\xi,1)=1 $ and $0 \leq (\xi, \eta ) \leq d(\eta ) $ for all $\eta \in S$, and check whether it is symmetric with the determinants of the leading principal minors all non-negative (which is a necessary condition for the matrix to be positive semi-definite).\\

2) For each symmetric fusion matrix $L^\xi$ whose leading principal minors have non-negative determinants, compute the set of decompositions $L^\xi=A^TA$ for $A$ a non-negative integer matrix.  

\end{algorithm}
\begin{remark}
 As in \cite[\S 3]{1102.2631} it is easier to decompose the reduced fusion matrices where the fusion rules of the identity element are left out; see the discussion there.
\end{remark}

Because of \ref{pglem}, the preceding algorithm finds all possible fusion matrices of fusion module basis elements.

Given a decomposition of a fusion matrix $L^\xi=A^TA$ we assign dimensions to the
rows of $A$ through the vector $d_A=\frac{1}{\sqrt{d( \xi)}}A d $, where $d$ is the Frobenius-Perron dimension vector of $F$. We define the dimension of $A$ to be $d(A):=\sqrt{d( \xi)} $.

We can now use this list of possible fusion matrices to find all fusion modules. The data of a fusion module is given by the collection of fusion matrices for each basis element, so we just need to check which collections of fusion matrices match up to give consistent module multiplication rules. By the Frobenius-Perron theory for fusion modules, a necessary condition for a collection of fusion matrices $(A^1,.., A^n)$ to give a fusion module is that all the $A^i$ have a common dimension vector $d_A$ of length $n$ and that $d_{A^i}=(d_A)(i)$. However, since we only keep a list of fusion matrices up to permutations of rows, to find all fusion modules from the list of fusion matrices we need to consider permutations of rows each fusion matrix. This is accomplished as follows.

\begin{algorithm}\label{alg1}
 To find all right fusion modules over a fusion ring $(F,S)$:\\

1) Find all possible fusion matrices for the basis elements, i.e. non-negative integer matrices $A$ such that  $A^TA=L^{\xi}$ for some $ \xi$ in $F$ (with coefficients bounded by $(\xi, \eta) \leq d(\eta), \ \forall \eta \in S $) as in the previous algorithm. Sort the rows of each matrix $A$ by increasing dimension.\\

2) For each $n$, find all $n$-tuples of fusion matrices $(A^1,.., A^n)$ such that all the $A_i$ have $n$ rows and share a common row dimension vector $d_A$ with $d(A_i)=d_A( i )$.\\

3) For each $n$-tuple of fusion matrices $(A^1,.., A^n)$ with common dimension vector $d_A$ such that $d(A^i)=d_A(i) $, we try all possible ways to build fusion modules as follows.

By a \textit{$k$-consistent set of permutations}, $k \leq n $, for $(A^1,.., A^n)$, we will mean a set of $k$ elements $\{\sigma_1,..,\sigma_k \} $ in the symmetric group $\mathcal{S}_n$, such that:

(i) $\sigma_i(d_A)=d_A, \ \forall 1 \leq i \leq k $, where the $\sigma_i $ act by permuting the entries of $d_A $;\\
(ii) $A^i_{ \sigma_i(j)l }=A^j_{ \sigma_j(i)\bar{l}} $ for all $1 \leq l \leq m $ and $1 \leq i,j \leq k $
(where the involution on the columns of $A_i$ is defined by the involution in the fusion ring.)

3a) We inductively build all possible $n$-consistent sets of permutations for 
$(A^1,.., A^n)$ by starting with the empty set (corresponding to $k=0$). 
Then for each $k$-consistent set of permutations for $k < n $, we find all ways to extend to a $(k+1)$-consistent set of permutations by checking which of those permutations $\sigma_{k+1}$ that satisfy (i) also satisfy (ii) for $i=k+1,\  j \leq k+1, \ 1 \leq l \leq m $.

3b) For all $n$-consistent sets of permutations for $(A^1,.., A^n)$, we check whether module associativity holds:
$$\sum_{1 
\leq i \leq m} \limits L^{s}_{it} A^{r}_{\sigma_r (l)i} = \sum_{1 \leq j \leq n}A^r_{ \sigma_r(j)s}  A^j_{ \sigma_j(l)t}, 
\ 1 \leq s,t \leq m, \ 1 \leq l,r \leq n.$$

\end{algorithm}

Suppose $(\sigma_1,...,\sigma_n)$ is an $n$-consistent sets of permutations for $(A^1,...,A^n) $ satisfying
condition 3b) of Algorithm \ref{alg1}. Then we can construct a right fusion module $(M,T) $ with  $T=\{ \eta_1,...,\eta_n \} $ 
over $(F,S) $ by setting
$$\eta_j \xi_i =\sum_{k=1}^{n} A^j_{k \sigma_j(i)}  \eta_k  .$$
Conversely, any fusion module $(M,T) $ appears this way in the algorithm (when $ \sigma_j(A^j)$ is the fusion matrix of  $\eta_j $ for some ordering of $T=\{\eta_1,...,\eta_n  \}$ and each $1 \leq j \leq n$). Therefore Algorithm \ref{alg1} gives the complete list of right fusion modules. 

\begin{remark}
 \begin{enumerate}
  \item The list of fusion modules generated by Algorithm \ref{alg1} may contain duplicates. However it is easy to check whether two given fusion modules
  are isomorphic (one simply looks at all the dimension preserving bijections on the basis sets and checks whether 
  the multiplicative structure constants are preserved under that bijection), so one can go down the list and eliminate duplicates. 
  \item A non-trivial permutation may act trivially on a matrix $A_i$ if $A_i$ has multiple identical rows, 
  so one can save time when building the fusion modules by only considering at each stage those permutations $\sigma$ for which $\sigma(A_i)$ is distinct
  from $\sigma'(A_i) $ for all previously checked permutations $\sigma'$.
 \end{enumerate}

\end{remark}

If $(M,T) $ is a right fusion module over $(F,S)$, one can define a left fusion module $(N,T) $ on the same basis sets by 
defining $\xi \eta=\eta \bar{\xi}, \ \forall \eta \in T, \ \xi \in S $. This construction gives a bijection between the right and 
left fusion modules over $(F,S)$. Thus once one has the right fusion modules, one also has the left fusion modules.

If $(M,T) $ is a right fusion module over $(F,S)$, with orderings $S =\{\xi_1,...,\xi_m \}$ and 
$T=\{\eta_1,...,\eta_n \} $, we denote the corresponding fusion matrix by
by $M^{k}_{ij} := M^{\eta_k}_{\eta_i \xi_j}, \ 1 \leq i,k \leq n, \ j \leq i \leq m $, with analogous notation for left fusion modules. We set $\bar{j}=l $, where $\xi_j=\bar{\xi}_l$.

%

Once we have the fusion modules we can find the fusion bimodules. The data of a fusion bimodule consists of a right fusion module, a left fusion module, and an identification of the bases of these modules. By Lemma \ref{lem:leftright}, it is necessary that the two bases have a common dimension vector and  that the identification perserves this vector. Therefore, if the bases of the left and right fusion modules are indexed by $\{1,...,n \} $ with increasing Frobenius-Perron dimensions, then an identification of the bases can be expressed as a permutation in $\mathcal{S}_n$ which preserves the common dimension vector.

\begin{algorithm} \label{alg2}
 To find the bimodules over a pair of fusion rings $(F_1,S_1)$ and $(F_2,S_2) $:\\

1) Find all pairs $((M,T_1),(N,T_2))$ such that $(M,T_1)$ is an left fusion module over $(F_1,S_1)$ and 
$(N,T_2)$ is an right fusion module over $(F_2,S_2)$ and such that $T_1$ and $T_2$ have the same dimension vector $d_T$.

2) For each such pair, fix orderings on $T_1$, $T_2$, $S_1$, and $S_2 $ and let $M^k_{ij}$ and $N^{p}_{qs} $ be the multiplicative
structure constants for $M$ and $N$ respectively. Let $m $, $n$, and $r$ be the sizes of $S_1 $, $T_1$, and $ S_2$, respectively.
For each dimension preserving permutation $\sigma$ of the common dimension vector $d_T$ of $T_1 $ and $T_2$,
check whether we have, when $N $ is twisted by $ \sigma$,

(i) Bimodule associativity: $$\sum_{p=1}^{n}  \limits M^i_{jp} N^{\sigma(p)}_{\sigma(q)l}=\sum_{p=1}^n \limits N^{\sigma(j)}_{\sigma(p)l} M^p_{qi}
, \  1 \leq j,q \leq n, \ 1\leq i \leq m, \ 1 \leq l \leq r .$$

and

(ii) Dual associativity: $$\sum_{i = 1}^m \limits M^k_{ji}M^q_{pi}=\sum_{l=1}^r \limits N^{\sigma(p)}_{\sigma(j)l}N^{\sigma(q)}_{\sigma(k)l}, \ 1 \leq j,k,p,q \leq n  .$$

%
%
%
%
\end{algorithm}

Given an output $((M,T_1), (N,T_2),\sigma ) $ of Algorithm \ref{alg2} satisfying conditions
(i) and (ii) in Step 2, we can construct a fusion bimodule on a basis of size $\#(T_1) =\#(T_2)$ by defining the left
fusion module structure by $M^k_{ij} $ and the right fusion module structure by $N^{\sigma(p)}_{\sigma(q)s} $. Conversely, any fusion bimodule will be found this way since the constituent
left and right fusion modules must appear on the lists of left and right fusion modules found in Algorithm \ref{alg1} (the permutation $\sigma $ is necessary since the lists of left and right fusion modules only contain one representative of each isomorphism class.)

Once we have lists of fusion bimodules over various fusion rings, we would like to check which triples of fusion bimodules are multiplicatively compatible. Let $({}_A K {}_B, {}_B L {}_C, {}_A M {}_C) $ be a triple of fusion bimodules, with bases $\xi_i$ for $ 1 \leq i \leq l$,  $\eta_j$ for $ 1 \leq j \leq m$, and  $\mu_k$ for $1 \leq k \leq n $, respectively. We will find possible multiplication maps for this triple by inductively defining multiplication for pairs of basis elements $(\xi_i,\eta_j)$, until we have defined multiplication for all pairs of basis elements.  For each $(i,j) $, the product $\xi_i \eta_j $ will be a non-negative combination of the $\mu_k $ whose dimension is the product of the dimensions of $\xi_i $ and  $ \eta_j$. At each step of the induction, we check consistency of our choice for $\xi \eta $ with all previously defined $\xi_i \eta_j$ using Lemma \ref{inthomass}. The details are as follows.

We order the set of pairs of integers $(p,q) , \ 1 \leq p \leq m, 1 \leq q \leq n$ lexicographically, and for each $(p,q) $ we let $(p,q)' $ denote the successor of $(p,q)$ in this order. 

\begin{definition}
 A $(p,q)$-partial multiplication map for $(p,q) \leq (m,n) $ is an assignment of a vector of integers of length $n$, $v^{ij}$, for each pair $(i,j) \leq (p,q)$ such that:\\
(a) $d(\xi_i)d( \eta_j) = \sum_{k=1}^n \limits v^{ij}_kd( \mu_k)$\\
(b) $(\bar{x}_ix_i,y_j\bar{y}_j) = \sum_{k=1}^n \limits (v^{ij}_k)^2$ for all $  (i_1,j_1),(i_2,j_2) \leq (p,q) $\\
(c) for all basis elements $\lambda \in A$ and all  $(i_1,j_1),(i_2,j_2) \leq (p,q)$ we have $$(( \xi_{i_1} (\eta_{j_1} \bar{\eta}_{j_2})) \bar{\xi}_{i_2}, \lambda)=\sum v^{i_1 j_1}_{k_1} v^{i_2j_2}_{k_2} (\mu_{k_1} \bar{\mu}_{k_2}, \lambda ) .$$\\

(c') for all basis elements  $\kappa \in C$ and all $ (i_1,j_1),(i_2,j_2) \leq (p,q)$
 we have $$( \bar{\xi}_{i_1}((\bar{\eta}_{j_1} {\eta}_{j_2}) {\xi}_{i_2}), \kappa)=\sum v^{i_1 j_1}_{k_1} v^{i_2j_2}_{k_2} (\bar{\mu}_{k_1} {\mu}_{k_2}, \kappa ) .$$\\

(d) for all  $(i_1,j_1),(i_2,j_2) \leq (p,q)$ we have
$$ v^{i_1 j_1} \cdot  v^{i_2 j_2} =  (\bar{\xi}_{i_1} \xi_{i_2},  \eta_{j_2} \bar{\eta}_{j_1}) .$$
\end{definition}
 We also define a $(0,0)$-partial multiplication map to be the empty map, and let  $(0,0)'=(1,1) $. 
 
 Note that a multiplication map in the sense of the previous section is in particular an $(m,n) $-partial multiplication map. Conditions (c) and (c') express the associativity between multiplication of elements in $A$ and $B$ and multiplication by duals in $A$ and $B$; if the partial multiplication map can be extended to a multiplication map which is a decategorification of the relative tensor product of invertible bimodule categories, then this associativity is guaranteed by Lemma \ref{inthomass}.

We can now find all multiplication maps by inductively building all $(p,q)$-partial multiplication maps, at each stage checking conditions (c),(c'), and (d). Note that a paticular consequence of (d) is that $$  v^{p q} \cdot  v^{pq} =  (\bar{\xi}_{p} \xi_{p},  \eta_{q} \bar{\eta}_{q}), \ \forall p,q ,$$
which determines the sum of the squares of the entries of each $v^{pq} $.

\begin{algorithm} \label{alg3}
To check whether a given triple of fusion bimodules is multiplicatively compatible:

 Step 1: Start with the $(0,0)$-partial multiplication map. Then inductively find all extensions of a given $(p,q)$-partial multiplication map, $(p,q) < (m,n) $, to a $(p,q)' $-partial multiplication map as follows:\\
 
 1a)  Let $(p,q)'=(p',q') $. Find candidates for $v^{p'q'} $ by checking 
  conditions (a) and (b) above as follows: for each sum of squares decomposition $(a_i,b_i)$ of  $(\bar{\xi}_{p'}\xi_{p'},\eta_{q'}\bar{\eta}_{q'})$ such that $\sum b_i \leq n$, form the vector $v$ of size $n$ given by $b_i$ copies of each $a_i$ with the rest of the entries equal to $0$. 
  Then find all distinct vectors $v'$
  which arise as permutations of $v$. \\ 
  
  1b) For each candidate $v'$ for for $v^{p'q'} $ found in (1a), check whether  $v' \cdot d_M = d(\xi_i)d( \eta_j) $, 
  where $d_M$ is the dimension vector of the bimodule $M$. Finally if the dimension condition is satisfied for $v' $, 
  check conditions (c),(c'), and (d) for  $(i_1,j_1)=(p',q') $ and all $(i_2,j_2) \leq (p',q') $, using  $v^{p'q'} =v'$.\\
%
%
%

Step 2: For each $(m,n) $-partial multiplication map found in Step 1, check whether 
$$ (\xi \rho) \eta= \xi (\rho \eta), \quad \lambda (\xi \eta)=(\lambda \xi) \eta, \quad (\xi \eta) \kappa=
\xi (\eta \kappa) ,  $$\\
for all basis vectors $$\lambda \in A, \ \rho \in B, \ \kappa \in C, \ \xi \in K, \ \eta \in L ,$$
where multiplication between elements of $K$ and $L$ is defined on basis elements by the partial multiplication map and extended biadditively.
\end{algorithm}

Any $(m,n)$-partial multiplication map found in Algorithm \ref{alg3} which satisfies the condition
in Step 2 is a multiplication map in the sense of the previous section, and conversely any multiplication map will be found this way. In fact for our purposes we do not need to find all possible multiplication maps for multiplicatively compatible triples; we only need to know for each triple whether at least one multiplication map exists. Once we find one we can terminate the algorithm for that triple.

In a similar way, we can check whether a given right module/bimodule/right module triple is multiplicatively compatible, using only condition (c') and not (c).

\section{The Brauer-Picard groupoid of the Asaeda-Haagerup categories}

In this section we compute all Morita equivalences between the three Asaeda-Haagerup categories and prove several results concerning the full Brauer-Picard groupoid.

\subsection{Fusion modules and bimodules of the Asaeda-Haagerup fusion rings}

The Grothendieck rings of the Asaeda-Haagerup fusion categories will be called the Asaeda-Haagerup fusion rings.

\begin{theorem}\\
(a) Up to isomorphism, the fusion modules over the Asaeda-Haagerup fusion rings are classified as follows: there are $24$ $AH_1$-modules, $21$ $AH_2$-modules, and $20$  $AH_3$-modules. Full multiplication rules for all these modules are included in the supplementary files \textit{AH1Modules}, \textit{AH1Modules}, and \textit{AH3Modules}.\\
 (b) Up to isomorphism, the fusion bimodules over the Asaeda-Haagerup fusion rings are classified as follows: there are $14$ $AH_1-AH_1$-bimodules, $13 $  $AH_2-AH_2$-bimodules, $13$  $AH_3-AH_3$ bimodules, $9$ $AH_1-AH_2$ bimodules, $7$  $AH_1-AH_3$ bimodules, and $6$  $AH_2-AH_3$ bimodules. Full multiplication rules for all these bimodules are included in the supplementary file \textit{Bimodules}.
\end{theorem}
\begin{proof}
Apply the algorithms from the previous section.
\end{proof}

For all pairs $(x,y)$ where $x$ is an fusion $AH_i-AH_j$-bimodule and $y$ is an fusion $AHj-AHk$-bimodule for some $1 \leq i, j, k \leq 3 $ we have computed the list of multiplicatively compatible triples $(x,y,z)$. These lists are included in the supplementary file \textit{BimoduleCompatibility}.

Similarly, for all pairs $(x,y)$ where $x$ is an fusion right $AHi$-module and $y$ is an fusion $AH_i-AH_j$-bimodule for some $1 \leq i, j\leq 3 $ we have computed the list of multiplicatively compatible triples $(x,y,z)$. These lists are included in supplementary file \textit{ModuleCompatibility}.


\subsection{Reductions using compatibility}

We now analyze the Brauer-Picard groupoid using a sequence of deductions from known information and multiplicative compatibility.

We identify the fusion rings $AH_1, AH_2, AH_3$ with the Grothendieck rings of the fusion categories $\mathscr{AH}_1,\mathscr{AH}_2 ,\mathscr{AH}_3   $. We will say that an $AH_i$ fusion module (resp. $AH_i-AH_j$ fusion bimodule) is realized if it is induced by an $\mathscr{AH}_i$-module category (resp. $\mathscr{AH}_i-\mathscr{AH}_j$-bimodule category.) We will say such a fusion module (resp. fusion bimodule) is realized uniquely if any two categories which realize it are equivalent as module categories (resp. bimodule categories). Any fusion bimodule has a dual bimodule; if the original fusion bimodule is realized by a bimodule category then the dual bimodule is realized by the opposite category. 

We will be referring extensively to the lists of fusion modules, fusion bimodules, and multiplicatively compatible triples over the Asaeda-Haagerup fusion rings. These lists are given in the supplementary data files \textit{AH1Modules}, \textit{AH2Modules}, \textit{AH3Modules}, \textit{Bimodules}, \textit{BimoduleCompatibility}, and \textit{ModuleCompatibility}.

We introduce the following notation. The symbol $a_{i}$ will denote the $a^{th}$ fusion module on the list of $AH_i$ fusion modules in the supplementary file \textit{AH1Modules}, \textit{AH2Modules}, or \textit{AH3Modules} (depending on whether $i=1,2$, or $3$.). Similarly, $a_{ij}$ will denote the $a^{th}$ fusion bimodule on the list of $AH_i-AH_j$ fusion bimodules. For two fusion bimodules $a_{ij}$ and $b_{jk} $, $a_{ij} \cdot b_{jk} $ will denote the set of bimodules $z_{ik} $ such that the triple $(a_{ij}, b_{jk}, z_{ik}) $ appears on the list of triples which are multiplicatively compatible. If there is a unique such $c_{ik}$, we say that $a_{ij}$ and  $b_{jk} $ have a unique multiplication and write $a_{ij} b_{jk}=c_{ik} $. The same notation is also used when $a_i$ is a fusion module and $b_{ij} $ is a fusion bimodule.

\begin{lemma}
 The following fusion bimodules are realized:\\
 $12_{11}, 14_{11}, 9_{12}, 6_{13}, 9_{21}, 13_{22},  6_{23}, 6_{31},  6_{32}, 13_{33} $.
\end{lemma}
\begin{proof}
The bimodules $14_{11}, 13_{22}, 13_{33}$ are realized by the trivial autoequivalences of $\mathscr{AH}_1, \mathscr{AH}_2, \mathscr{AH}_3$, respectively. The bimodules $9_{12},6_{13}, 12_{11} $ are realized by the $AH$, $AH+1$, and $AH+2$ subfactors, respectively. The bimodule $6_{23} $ is realized by the two-dimensional algebra objects in $\mathscr{AH}_2 $ and $\mathscr{AH}_3 $ by Theorem \ref{index2}. The bimodules $9_{21}, 6_{31}, 6_{32} $ are the dual bimodules to  $9_{12}, 6_{13}, 6_{23} $.
\end{proof}

\begin{lemma}
 The following fusion bimodules are realized:\\
 $10_{11}, 13_{11}, 2_{12}, 5_{12}, 8_{12}, 2_{13}, 3_{13}, 7_{13}, 2_{21}, 5_{21}, 8_{21}, 11_{22}, 12_{22},  1_{23}, 2_{31}$,$3_{31}, 7_{31},  1_{32}, 8_{33}, 12_{33}$. 
\end{lemma}

\begin{proof}

If two fusion bimodules $a_{ij}$ and $b_{jk}$ are realized and we have a unique multiplication $a_{ij} b_{jk}=c_{ik} $ then $c_{ik} $ is realized as well. We have the following unique multiplications (at each step we use the results of the previous step as inputs):

1) $12_{11} 9_{12} = 8_{12}$, $12_{11} 6_{13} = 2_{13}$, $9_{12} 6_{23} = 7_{13}$,  $6_{13}6_{32}=5_{12}$.

2) $12_{11} 5_{12} = 2_{12}$, $ 12_{11} 7_{13} = 3_{13}$.

Note that the dual bimodules $2_{21}, 5_{21},  8_{21},  2_{31}, 3_{31}, 7_{31} $ are realized as well.

3) $2_{13} 3_{31} = 13_{11}$, $ 2_{21} 5_{12} = 12_{22}$, $ 5_{21} 9_{12} = 11_{22}$, $ 2_{21}7_{13}=1_{23}$, $ 2_{31} 3_{13} = 12_{33}$, $ 3_{31}7_{13}=8_{33} $

Again the dual bimodule $1_{32}$ is also realized.

4) $12_{11} 13_{11}=10_{11}$.

\end{proof}

\begin{lemma} \label{unique_rel}

(a) Let $a_{ij}b_{jk}=c_{jk} $ be a unique multiplication. If $a$ and $c$ are realized uniquely and $b$ is realized, then $b$ is realized uniquely.  The same result holds interchanging the roles of $a$ and $b$ or for a unique multiplication $a_{i}b_{ij}=c_{j}$.

(b) Let $a_{ij} \bar{a}_{ij} = Id_{ii}$ be a unique multiplication, where $\bar{a}_{ji} $ is the dual fusion bimodule to $a_{ij}$, and $Id_{ii} $ is the trivial bimodule for $AH_i, i=1,2$. Then if $a_{ij}$ is realized, it is realized uniquely.
\end{lemma}

\begin{proof}
 (a) Suppose $a_{ij}b_{jk}=c_{jk} $ is a unique multiplication with $a$ and $c$ realized uniquely, and suppose that $b$ is also realized.  Let $\mathscr{A}, \mathscr{B}, \mathscr{C}$ be bimodule categories realizing $a,b,c $ respectively. Then we have $\mathscr{A} \otimes \mathscr{B} \cong \mathscr{C}$, so $\mathscr{B} \cong \mathscr{A}^{-1} \otimes \mathscr{C} $; since $a,c$ are uniquely realized, $\mathscr{A}, \mathscr{C} $ are uniquely determined, and therefore so is $\mathscr{B} $. The proofs of the other statements are similar.

(b) Suppose $a_{ij} \bar{a}_{ij} = Id_{ii}$ is a unique multiplication, where $i=1,2$. Since $\mathscr{AH}_1,\mathscr{AH}_2 $ have no outer automorphisms (Theorem \ref{no_out}), $Id_{ii} $ is realized uniquely. Suppose $\mathscr{A}_1,\mathscr{A}_2 $ are two realizations of $a$. Then by the unique multiplication of $a\bar{a} $ and the unique realization of $Id_{ii} $ we have $\mathscr{A}_1 \otimes (\mathscr{A}_2)^{-1} \cong \mathscr{Id}_i $, where $\mathscr{Id}_i$ is the trivial auto-equivalence of $\mathscr{AH}_i $; therefore we have $\mathscr{A}_1 \cong \mathscr{A}_2 $. So if $a$ is realized it must be realized uniquely.   
\end{proof}

\begin{lemma}
 The following fusion bimodules are realized: $ 8_{22}, 11_{33}$. 
\end{lemma}

\begin{proof}
 We know that $11_{22}$ and $12_{22}$ are realized. Moreover, since $11_{22} 11_{22}=13_{22}$ is a unique multiplication, $11_{22}$ is realized uniquely. Since we have $11_{22} 13_{22}=11_{22} $, we cannot also have  $11_{22}$ realized from $11_{22} \cdot 12_{22} $.  Therefore at least one member of $11_{22} \cdot 12_{22} = \{ 8_{22},9_{22},10_{22}, 11_{22} \} $ distinct from $11_{22}$ must be realized. However, $9_{22}$ cannot be realized since $11_{22} \cdot 9_{22}$ is empty, and $10_{22}$ cannot be realized since $10_{22} \cdot 11_{22}$ is empty. Therefore $8_{22}$ is realized.

Similarly, we know that $8_{33}$ and $12_{33}$ are realized, and since $8_{33}8_{33}=13_{33}$ is a unique multiplication, $8_{33}$ is realized uniquely. Therefore at least one member of $8_{33} \cdot 12_{33} = \{ 8_{33},9_{33},10_{33}, 11_{33} \} $ distinct from $8_{33}$ must be realized. Since $8_{33} \cdot 10_{33}$  and  $9_{33} \cdot 8_{33}$ are empty, $11_{33} $ must be realized.
\end{proof}

\begin{lemma} \label{23real}
 The fusion bimodules $1_{23}, 2_{23}, 4_{23}, 6_{23}$, and their duals\\
 $1_{32}, 2_{32}, 4_{32}, 6_{32}$ are realized uniquely. The bimodules $2_{23} $ and $5_{23} $ are not realized.
\end{lemma}

\begin{proof}
 We already know that $1_{23}$ and $6_{23}$ are realized. Since $5_{12}$ is realized and $5_{12}5_{23}$ is empty, $5_{23}$ cannot be realized. Since $3_{13} $ is realized $3_{13}3_{32} $ is empty, $3_{32}$ and its dual $3_{23}$ cannot be realized. Since $5_{21} $ and $7_{13}$ are realized and $5_{21} \cdot 7_{13} = \{ 2_{23}, 3_{23}  \} $, and $3_{23}$ is not realized, $2_{23} $ must be realized. Since $8_{21} $ and $7_{13}$ are realized and $8_{21} \cdot 7_{13} = \{ 4_{23}, 5_{23}  \} $, and $5_{23}$ is not realized, $4_{23} $ must be realized. 

Uniqueness of the realizations of $3_{23}, 4_{23}, 6_{23}$ follows from Lemma \ref{unique_rel} (b). For uniqueness of the realization of $1_{23}$ we must use Lemma \ref{unique_rel} several times. First, $7_{13}$ and its dual $7_{31}$ are uniquely realized by Lemma \ref{unique_rel} (b). Second, we have a unique multiplication  $12_{11}2_{12}=5_{12}$. Third, we have $2_{12} \cdot 2_{21} =  14_{11} $. 
Finally, we have the unique multiplication $1_{23}7_{31}=2_{21} $, so since $2_{21} $ and $7_{31}$ are uniquely realized, by Lemma \ref{unique_rel} (a), $1_{23} $ is uniquely realized as well.  
 \end{proof}

\begin{theorem}
(a) There are exactly $4$ invertible bimodule categories over each not-necessarily-distinct pair $\mathscr{AH}_i-\mathscr{AH}_j$, up to equivalence. These realize the following fusion bimodules, which are each realized uniquely:

$10_{11}, 12_{11}, 13_{11}, 14_{11}, 2_{12}, 5_{12}, 8_{12}, 9_{12},  2_{13}, 3_{13}, 6_{13}, 7_{13}$\\
$ 2_{21}, 5_{21}, 8_{21}, 9_{21}, 8_{22}, 11_{22}, 12_{22}, 13_{22}, 1_{23}, 3_{23}, 4_{23}, 6_{23}$\\
$ 2_{31}, 3_{31}, 6_{31}, 7_{31}, 1_{32}, 2_{32}, 4_{32}, 6_{32}, 8_{33}, 11_{33}, 12_{33}, 13_{33}$.

(b) The Brauer-Picard group of each $\mathscr{AH}_i$ is $\mathbb{Z}/2\mathbb{Z} \times \mathbb{Z}/2\mathbb{Z} $.

 \end{theorem}

\begin{proof}
 (a) We have already seen that all of these bimodules are realized. Since by Lemma \ref{23real} there are exactly $4$ bimodule categories between $\mathscr{AH}_2$ and $\mathscr{AH}_3$, there must be exactly $4$ bimodule categories between each pair  $\mathscr{AH}_i-\mathscr{AH}_j$. Therefore each of the bimodule categories on the list must be realized uniquely and there can be no others.

(b) This is immediate from (a) and the unique multiplications $10_{11} 10_{11} = 12_{11} 12_{11} = 13_{11} 13_{11} = 14_{11}$.
\end{proof}

\begin{corollary}
 $\mathscr{AH}_3$ does not admit any outer automorphisms.
\end{corollary}
\begin{proof}
 The Brauer-Picard group has order $4$ and we know that $4$ distinct fusion bimodules (which are also different as fusion modules) are realized by auto-equivalences; therefore there cannot be any outer automorphisms.
\end{proof}

\begin{lemma}\label{mod_uni}
(a) The fusion modules $2_1, 4_1, 5_1, 7_1, 10_1, 12_1, 13_1, 14_1, 15_1, 22_1, 23_1, 24_1$ are each realized uniquely. The fusion modules $3_1, 6_1,8_1,11_1, 17_1, 18_1, 19_1, 20_1$ are not realized.\\

(b) The fusion modules $1_2, 2_2, 5_2, 6_2, 7_2, 8_2, 10_2, 12_2, 14_2, 16_2, 20_2, 21_2 $ are each realized uniquely. The fusion modules $9_2,11_2,13_2, 15_2, 18_2 $ are not realized.\\

(c) The fusion modules $1_3, 4_3, 5_3, 6_3, 7_3, 8_3, 9_3, 12_3, 14_3, 17_3, 19_3, 20_3$ are each realized uniquely. The fusion modules $ 10_3,11_3,13_3,15_3$ are not realized.\\
\end{lemma}

\begin{proof}
 We prove (c). The first list of $12$ modules are realized by the $12$ $\mathscr{AH}_i-\mathscr{AH}_3$ bimodule categories above. For uniqueness, we use Lemma \ref{unique_rel}. First note that the trivial modules $22_1,21_2,20_3$ are realized uniquely. Also, the modules $12_1,13_1,23_1, 14_2, 16_2, 8_3, 15_3 $ are realized uniquely by the uniqueness of the $AH, AH+1,AH+2 $ and dimension $2$ algebra objects.

We have the following unique multiplications: $5_3 3_{31}=22_1, 6_3 3_{31}=12_1, 7_3 11_{33}=20_3, 9_3 6_{31}=23_1,  12_3 4_{32}=21_2,17_3 7_{31}=22_1, 19_3 7_{31}=12_1$, which proves the unique realization of $5_3, 6_3, 7_3, 9_3, 12_3, 17_3, 19_3$. That leaves $1_3$ and $4_3$.

The modules $6_2$ and $7_2$ are realized with the bimodules $8_{22}$ and $11_{22}$. We have unique multiplications $6_2 4_{23} = 9_3 $ and $7_2 5_{21} =23_1$, so the modules are realized uniquely. Since $1_3 \cdot 6_{32} = \{ 6_2, 7_2 \} $, this implies that any realization of $1_3$ must have dual category $\mathscr{AH}_2$. But we already know that $1_{23}$, which is the only $AH_2-AH_3$ bimodule extension of $1_3$, is realized uniquely, so $1_3$ is realized uniquely.

Similarly, the module $2_2$ is realized by the bimodule $2_{12}$, and by the unique multiplication $2_2 5_{21}=13_1 $ the realization is unique. Then from the unique multiplication $4_3 6_{32} = 2_2$, $4_3$ is realized uniquely.

The empty multiplication $10_3 \cdot 7_{31}, 11_3 \cdot 7_{31} $, $13_3 \cdot 7_{31} $,
and $15_3 \cdot 7_{31}$ imply that $ 10_3,11_3,13_3,15_3$ are not realized.

The proofs of (a) and (b) are similar, except easier since we can use the results of (c) when checking multiplicative compatibility.

\end{proof}

We end with a classification of possible other objects in the Brauer-Picard groupoid.

\begin{theorem}
 Let $\mathscr{C} $ be a fusion category which is Morita equivalent to the $\mathscr{AH}_i, i =1,2,3 $, but not isomorphic to any of them. Then exactly one of the following four cases holds:

(a) Every $\mathscr{C}-\mathscr{AH}_1 $ Morita equivalence realizes $9_1$, every $\mathscr{C}-\mathscr{AH}_2 $ Morita equivalence realizes $19_2$, and every $\mathscr{C}-\mathscr{AH}_3 $ Morita equivalence realizes $16_3$.

(b) Every $\mathscr{C}-\mathscr{AH}_1 $ Morita equivalence realizes $16_1$, every $\mathscr{C}-\mathscr{AH}_2 $ Morita equivalence realizes $4_2$, and every $\mathscr{C}-\mathscr{AH}_3 $ Morita equivalence realizes $18_3$.

(c) Every $\mathscr{C}-\mathscr{AH}_1 $ Morita equivalence realizes $21_1$, every $\mathscr{C}-\mathscr{AH}_2 $ Morita equivalence realizes $17_2$, and every $\mathscr{C}-\mathscr{AH}_3 $ Morita equivalence realizes $2_3$.

(d) Every $\mathscr{C}-\mathscr{AH}_1 $ Morita equivalence realizes $1_1$, every $\mathscr{C}-\mathscr{AH}_2 $ Morita equivalence realizes $3_2$, and every $\mathscr{C}-\mathscr{AH}_3 $ Morita equivalence realizes $2_3$.
\end{theorem}
\begin{proof}

First note that by Lemma \ref{mod_uni}, the only modules whose realizations are not yet known are $1_1, 9_1, 16_1, 21_1, 3_2, 4_2, 17_2, 19_2, 2_3, 3_3, 16_3, 18_3$. Let $\mathscr{C}$ be a fusion category which is Morita equivalent to the $\mathscr{AH}_i, i =1,2,3 $ but not equivalent to any of them. Then there is a $\mathscr{C}-\mathscr{AH}_3$ Morita equivalence which realizes one of the four modules $2_3,3_3, 16_3, 18_3$. We will show that the four possibilities correspond to the four cases in the statement of the theorem.

First, suppose that a $\mathscr{C}-\mathscr{AH}_3$ Morita equivalence realizes $16_3$, and consider the action of the Brauer-Picard group of $\mathscr{AH}_3$ on the bimodule category. Since the Brauer-Picard group realizes the bimodules $8_{33}, 11_{33}, 12_{33}, 13_{33}$, the unique multiplications $16_3 11_{33} = 16_3, 16_3 12_{33} = 16_3, 16_3 13_{33}=16_3 $, together with the multiplication $16_3 \cdot 8_{33} = \{ 5_3, 16_3 \} $, imply that every $\mathscr{C}-\mathscr{AH}_3$ Morita equivalence realizes $16_3$. Similarly, if a  $\mathscr{C}-\mathscr{AH}_3$ Morita equivalence realizes $18_3$, every $\mathscr{C}-\mathscr{AH}_3$ Morita equivalence must realize $18_3$.

For $2_3 $ and $3_3$ we have the following multiplications: $2_3 \cdot 8_{33} =\{ 2_3, 3_3, 18_3 \},  2_3  11_{33} = 2_3 , 2_3 \cdot 12_{33} = \{2_3, 3_3 \}, 2_3 13_{33}=2_3, 3_3 8_{33} =3_3,  3_3 \cdot  11_{33} = \{ 2_3, 3_3, 18_3 \} , 3_3 \cdot 12_{33} = \{2_3, 3_3 \}, 3_3 13_{33}=3_3, $. From these possibilities, it follows that if either $2_3$ or $3_3$ is realized by a $\mathscr{C}-\mathscr{AH}_3$ Morita equvalence then it is realized by multiple inequvalent $\mathscr{C}-\mathscr{AH}_3$ Morita equvalences. Therefore the only possibilities are either that $2_3$ or $3_3$ is realized by $4$ different $\mathscr{C}-\mathscr{AH}_3$ Morita equivalences, or that $2$ different $\mathscr{C}-\mathscr{AH}_3$ realize $2_3$ and another $2$ different $\mathscr{C}-\mathscr{AH}_3$ Morita equivalences realize $3_3$. Suppose the latter ocurred. From the unique multiplication  $2_3  11_{33} = 2_3$ we see that the nontrivial autoequivalence $11_{33} $ must permute the $2$ realization of $2_3$. Similarly, from  $3_3 \cdot 8_{33} =3_3$ the nontrivial autoequivalence $8_{33}$ must permute the $2$ realizations of $3_3$. But since there are no fixed points for the action of any nontrivial element of the Brauer-Picard group,  this means that $8_{33}$ and $11_{33}$ implement the same order $2$ permutation on the set of $\mathscr{C}-\mathscr{AH}_3$ Morita equivalences, which implies that they are inverses of each other. But since every autoequivalence in the Brauer-Picard group has order $2$ this is impossible.

We have now seen that if any $\mathscr{C}-\mathscr{AH}_3$ Morita equivalence realizes one of  $2_3,3_3, 16_3, 18_3$, then every  $\mathscr{C}-\mathscr{AH}_3$ realizes the same $AH3$ module. Similar arguments show the same thing for the  $4$ $AH_1$ modules and the $4$ $AH_2$ modules in the statement of the theorem. It remains only to check which $AH_i$ modules are compatible with which $AH_j$ modules. This too can be sorted out from the multiplicative compatibility lists. If $16_3 $ is realized, then the unique multiplications $16_3 6_{32} =  19_2 $ and $16_3 7_{31} =  9_1 $ show that $19_2$ and $9_1 $ must be realized as well. 

The other three cases are handled similarly.

\end{proof}

We can go quite a bit further in the classification of the first three cases in this theorem. Specifically, we have the following conjecture.

\begin{conjecture}
Cases (a)-(c) above are each realized by a unique fusion category, each having the following properties:

(a) There are $4$ invertible objects and $4$ objects of dimension $4+\sqrt{17}$.

(b) The Brauer-Picard group is implemented  by outer automorphisms.

\end{conjecture}

Note that property (b) is in sharp contrast to the situation with $\mathscr{AH}_1-\mathscr{AH}_3 $, which do not admit any outer automorphisms and whose Morita equivalences are also all distinct as module categories, and indeed, even as fusion modules.

We can reduce the proof of this conjecture to the construction of a single subfactor with the following principal graph (which we hope to address in future work).

$$\vpic{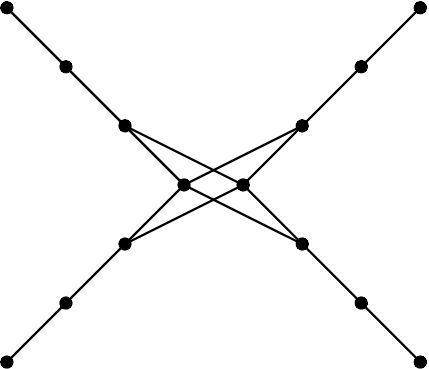} {2in}   $$

\section{Subfactors in the Asaeda-Haagerup family}

In this section we classify subfactors whose even parts are among the three Asaeda-Haagerup fusion categories, up to isomorphism of the planar algebra. 

\subsection{Principal graphs of small-index subfactors in the Asaeda-Haagerup categories}

We now list the principal and dual graphs of a minimal dimension generating object in each of the $24$ biomdule categories, which, up to duality and equivalence, exhaust the full subgroupid of the Brauer-Picard groupoid generated by $\mathscr{AH}_1 -\mathscr{AH}_3$. By a generating object of a module category over a fusion category we mean an object whose internal end tensor generates the fusion category. In this case of the Asaeda-Haagerup categories, this just excludes objects of dimensions $1$ and $\sqrt{2}$, whose internal end tensor generate trivial and $Vec_{\mathbb{Z}/2\mathbb{Z}}$ proper subcategories, respectively.
 
For the $ \mathscr{AH}_i-\mathscr{AH}_i$ bimodule categories, the two graphs are always the same, so we only list the common graph once. For the  $\mathscr{AH}_i-\mathscr{AH}_j, i \neq j$ bimodule categories, we give the ordered pair of principal graphs.

Multiplicities on edges are denoted by putting a number next to the edge, or, if the graph is too complex, putting a tuple of numbers next to a vertex.  In the latter case you should read the numbers as labeling the edges from top to bottom.

(a) $\mathscr{AH}_1-\mathscr{AH}_1$-categories: \\
\vpic{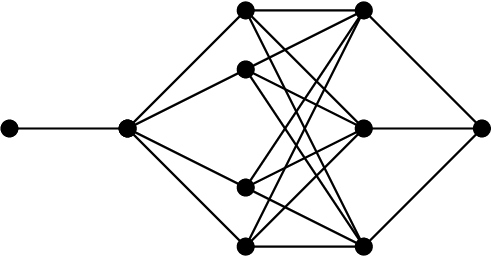} {1.9in} \vpic{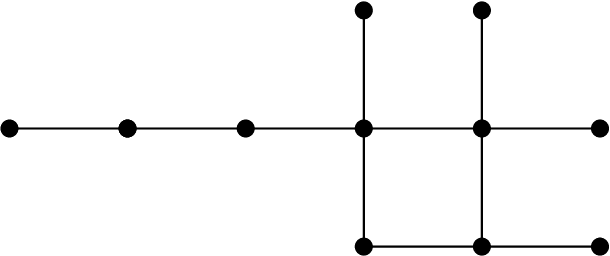} {2in} \\ 
\vpic{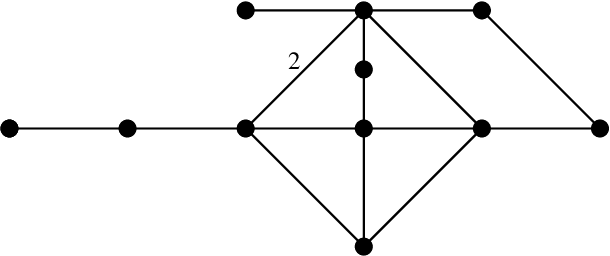} {2in} \vpic{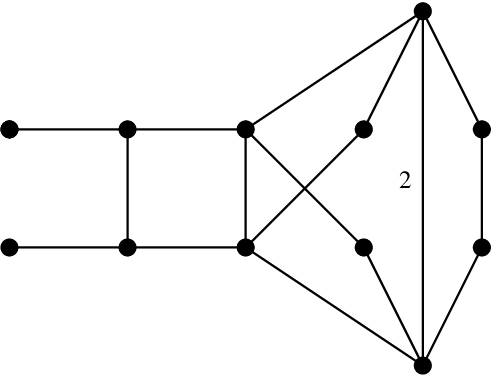} {1.8in} \\

(b) $\mathscr{AH}_2-\mathscr{AH}_2$-categories: \\

\vpic{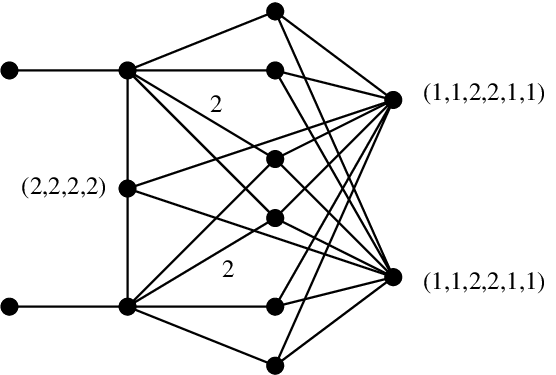} {1.9in} \vpic{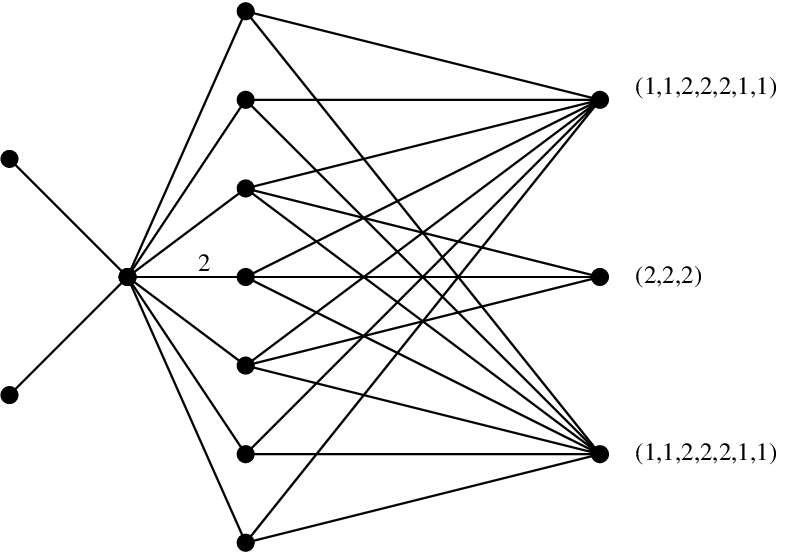} {2in} \\
\vpic{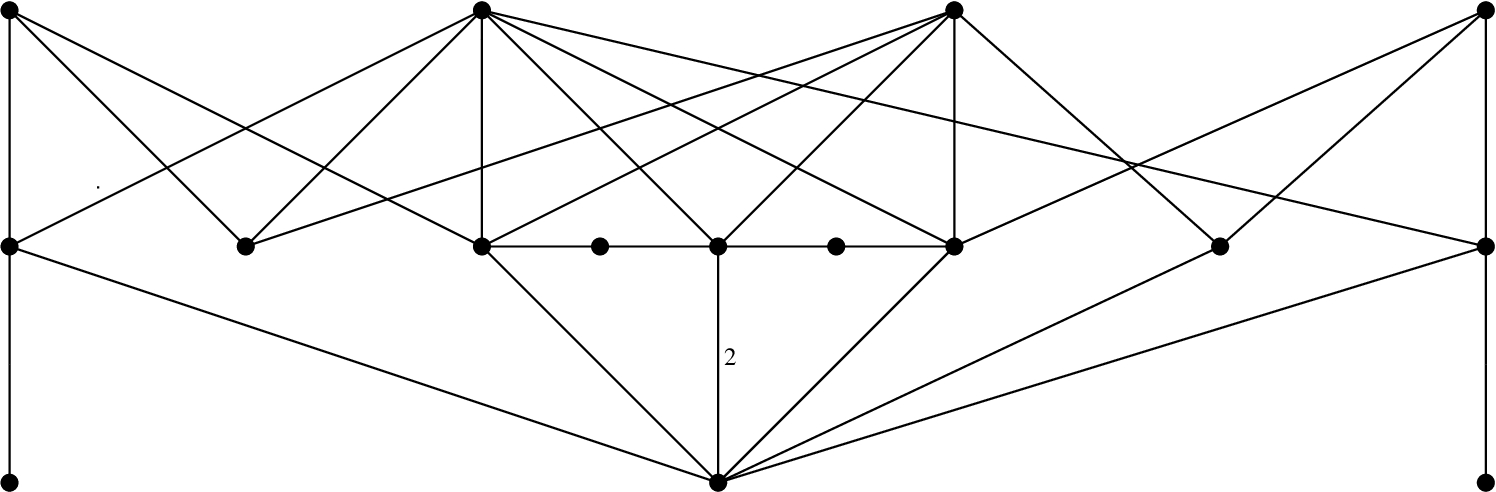} {2in} \vpic{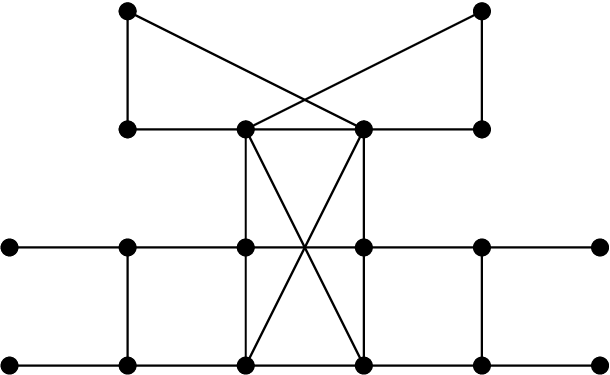} {1.9in} \\

(c) $\mathscr{AH}_3-\mathscr{AH}_3$-categories: \\

\vpic{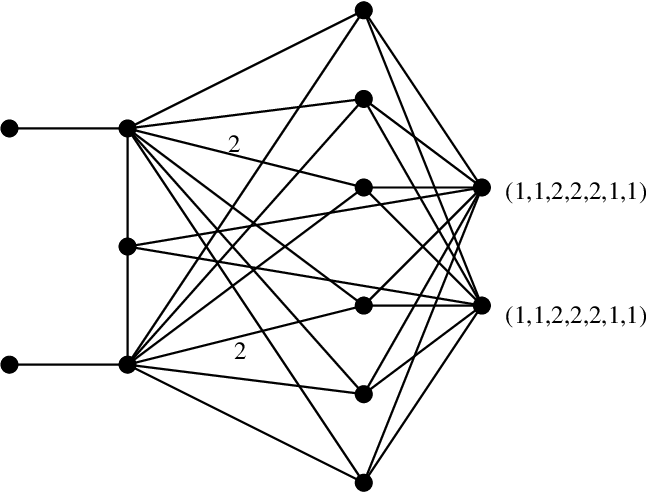} {2in} \vpic{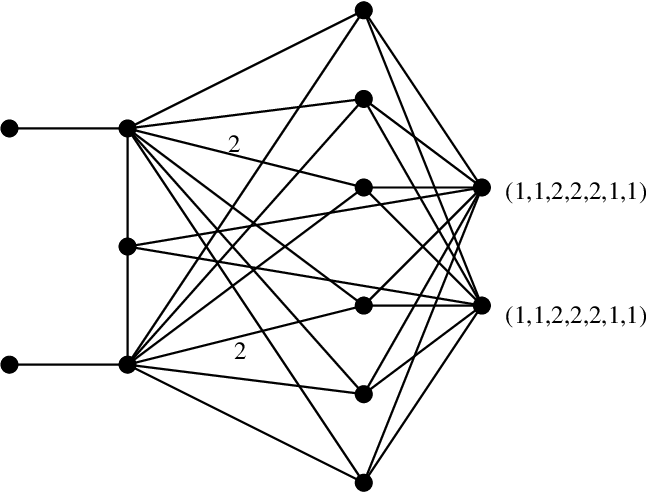} {2in} \\
\vpic{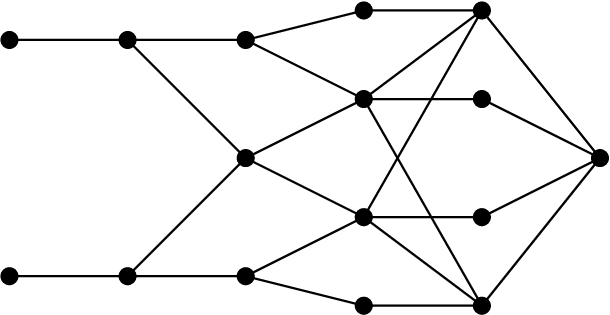} {2in} \vpic{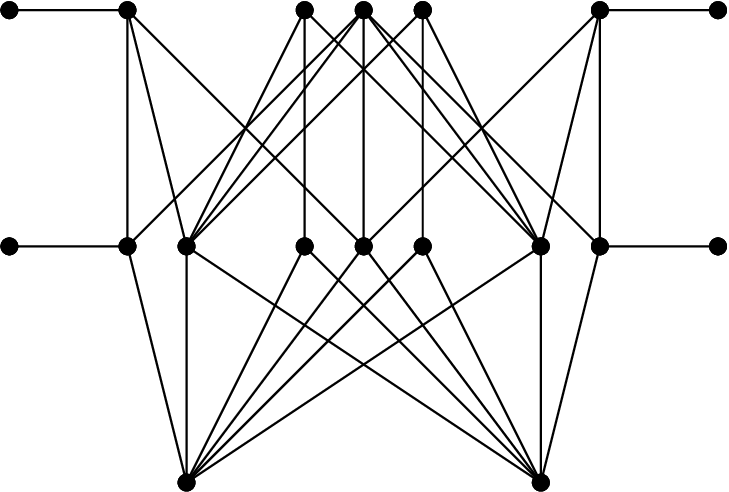} {2in} \\

(d) $\mathscr{AH}_1-\mathscr{AH}_2$-categories: \\

$\left( \vpic{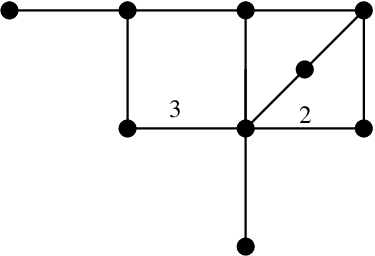} {1.3in} , \vpic{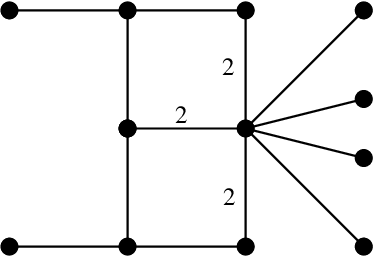} {1.3in} \right)$ \\
$\left( \vpic{g1_6} {1.6in} , \vpic{g2_5} {1.6in} \right)$ \\
$\left( \vpic{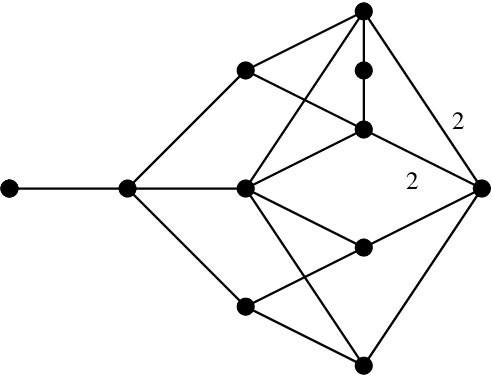} {1.6in} , \vpic{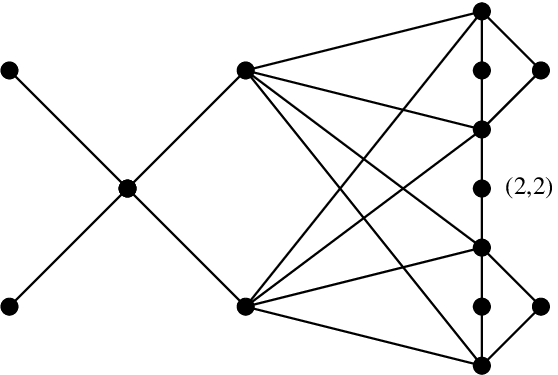} {1.6in} \right)$ \\
$\left( \vpic{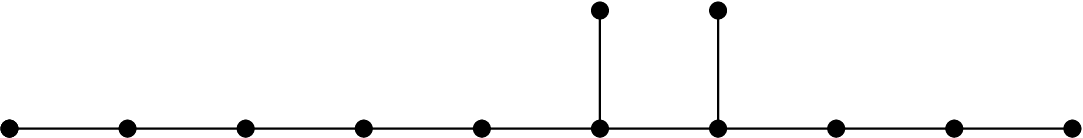} {2in} , \vpic{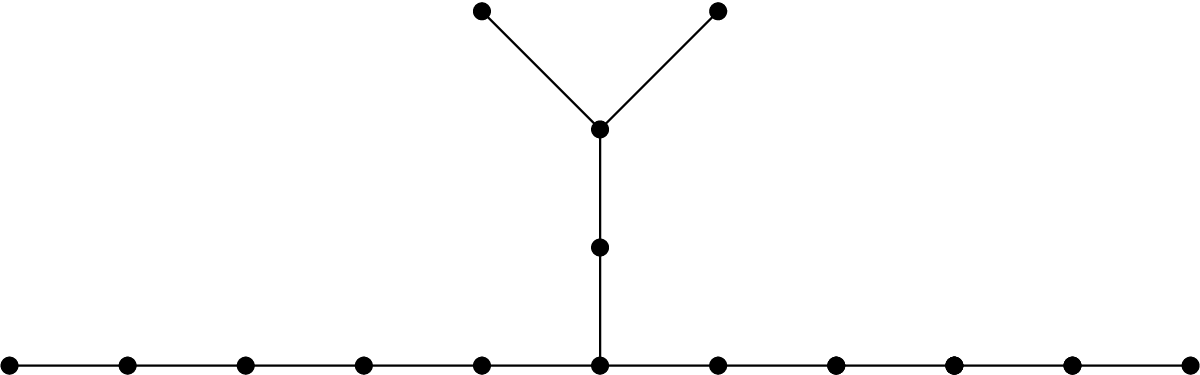} {2in} \right)$ \\

(e) $\mathscr{AH}_1-\mathscr{AH}_3$-categories: \\

$\left( \vpic{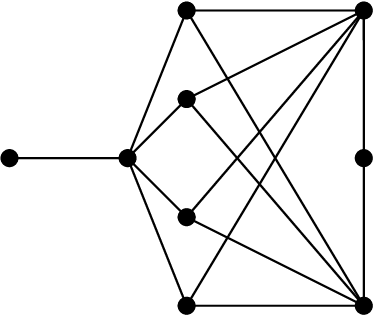} {1.3in} , \vpic{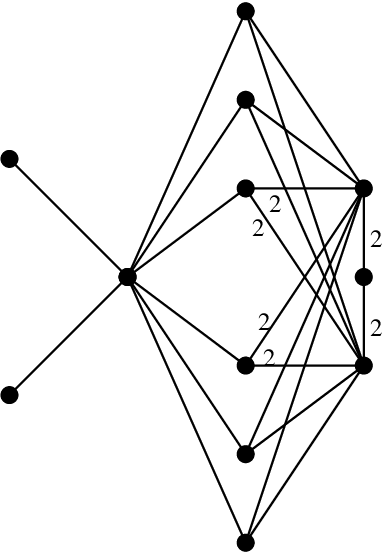} {1.3in} \right)$ \\
$\left( \vpic{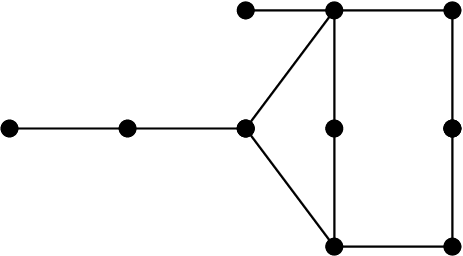} {1.5in} , \vpic{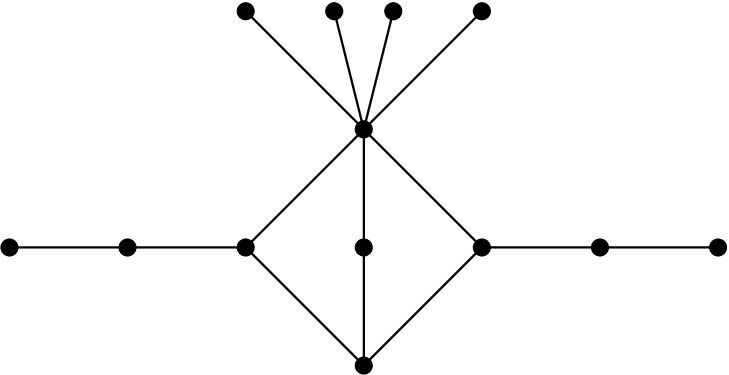} {2in} \right)$ \\
$\left( \vpic{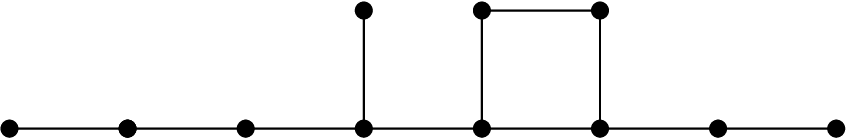} {2in} , \vpic{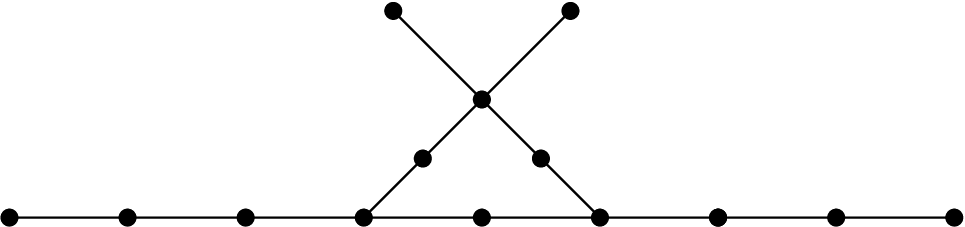} {2in} \right)$ \\
$\left( \vpic{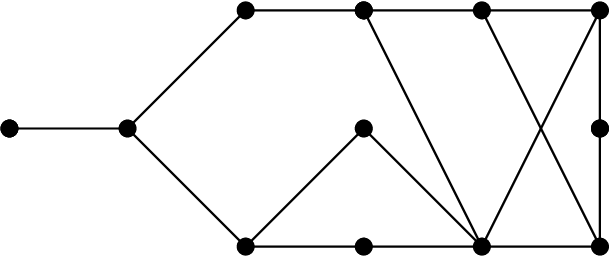} {1.9in} , \vpic{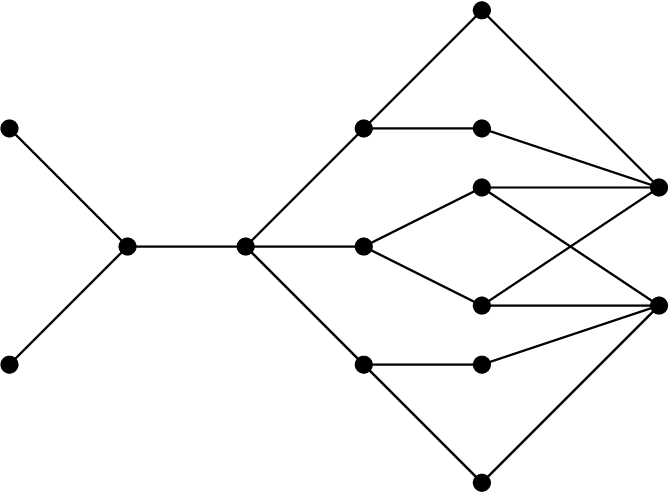} {1.9in} \right)$ \\

(f) $\mathscr{AH}_2-\mathscr{AH}_3$-categories: \\

$\left( \vpic{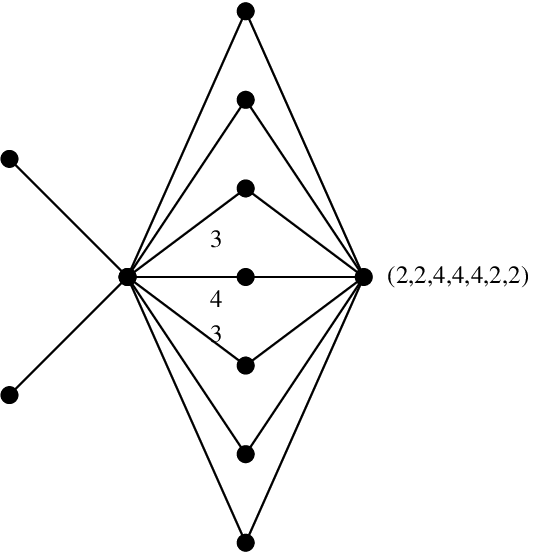} {2in} , \vpic{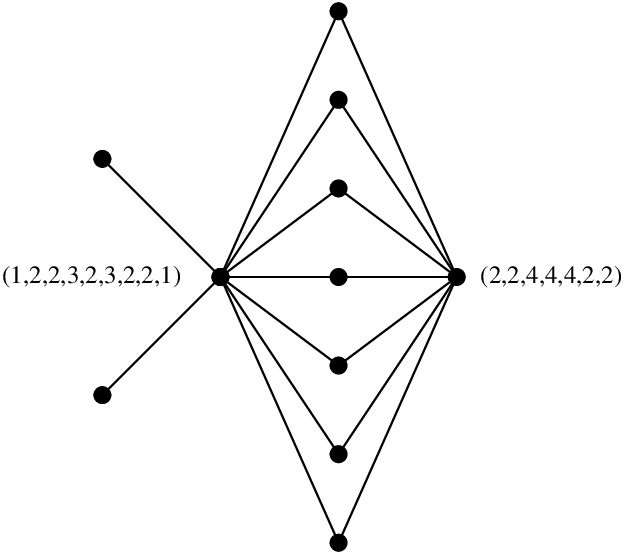} {2in} \right)$ \\
$\left( \vpic{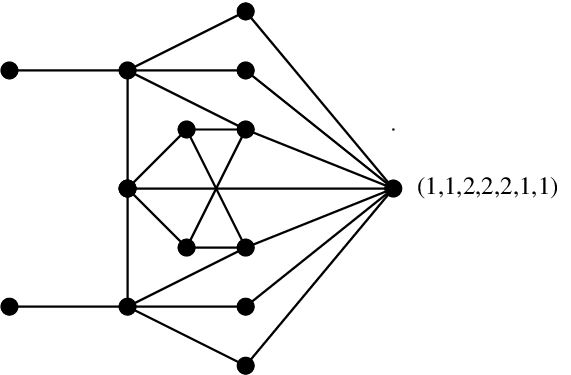} {2in} , \vpic{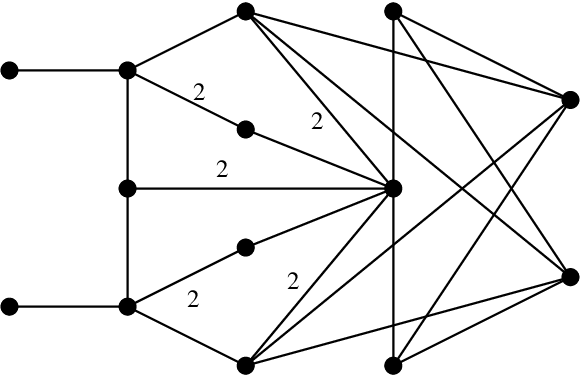} {2in} \right)$ \\
$\left( \vpic{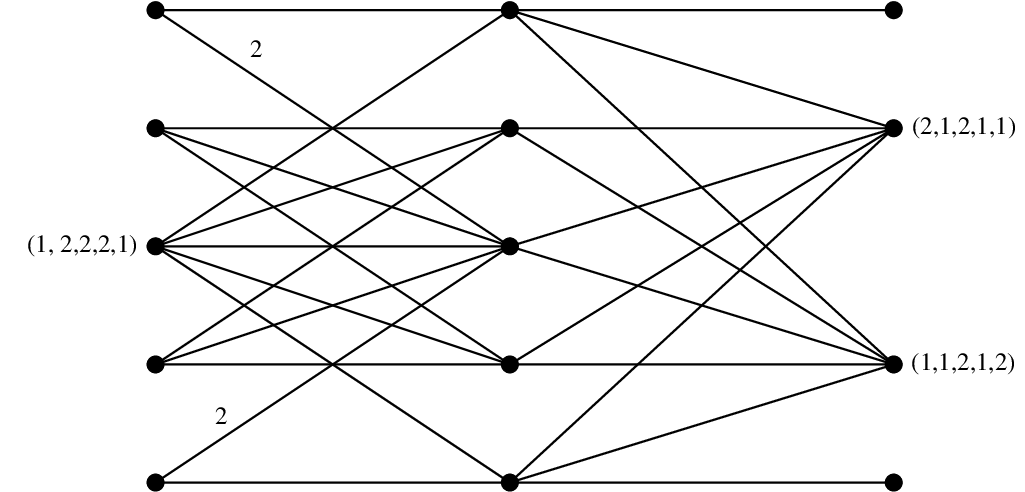} {2in} , \vpic{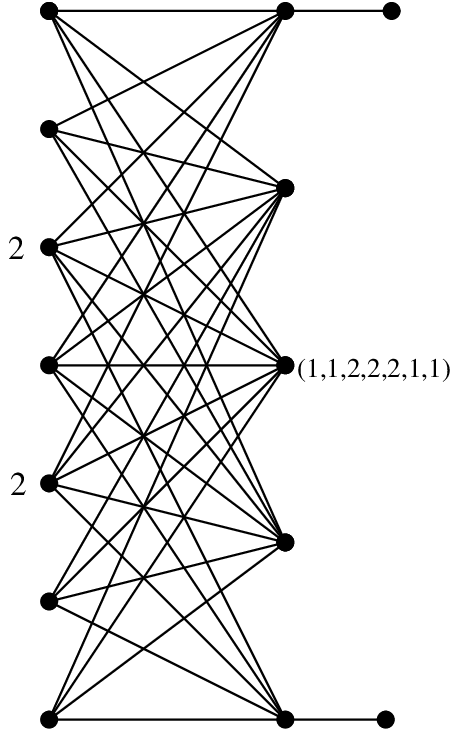} {1.2in} \right)$ \\
$\left( \vpic{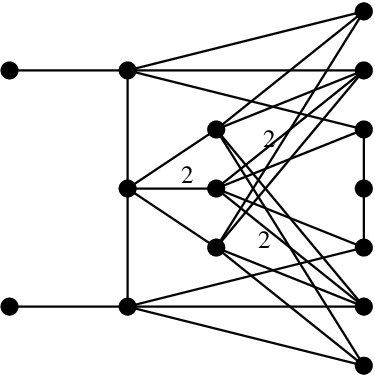} {1.3in} , \vpic{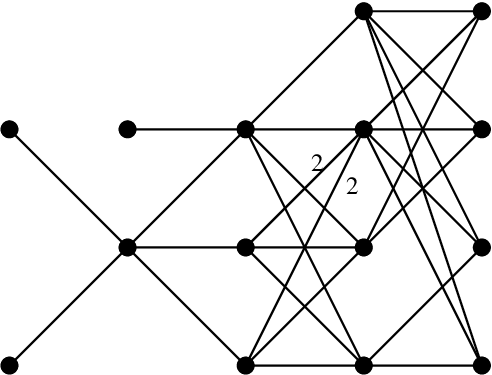} {1.5in} \right)$ \\

\subsection{Classification of subfactors}

\begin{theorem}
 There are $111$ distinct irreducible subfactors  ($76$ up to duality), up to isomorphism of the planar algebra, both of whose even parts are in the set $\{ \mathscr{AH}_1,\mathscr{AH}_2,\mathscr{AH}_3 \}$. The principal and dual principal graphs for all of these subfactors have been computed (they are actually part of the fusion bimodule data) and are available in the supplementary files \textit{Bimodules}.
\end{theorem}

\begin{proof}
 
The algebra of left internal end of any simple object in one of the $36$ bimodule categories in the groupoid above is an irreducible Q-system, from which one can construct a subfactor and associated planar algebra whose even parts belong to $\{ \mathscr{AH}_1,\mathscr{AH}_2,\mathscr{AH}_3 \}$. Clearly every planar algebra with even parts in $\{ \mathscr{AH}_1,\mathscr{AH}_2,\mathscr{AH}_3 \}$ arises this way. The question is when two different simple bimodule objects $\kappa$ and $ \lambda$ give the same planar algebra, which happens iff there is an automorphism of the left fusion category which takes the algebra of internal end of $\kappa$ to that of $\lambda$. For $\{ \mathscr{AH}_1,\mathscr{AH}_2,\mathscr{AH}_3 \}$, the only nontrivial automorphisms are the unique inner automorphisms of $\mathscr{AH}_2$ and $\mathscr{AH}_3 $ corresponding to the unique nontrivial invertible object in each case. So $\kappa$ and $\lambda$ give the same planar algebra iff there is an invertible object $\alpha $ in the left category such that $\alpha \kappa \bar{\kappa} \alpha  \cong \lambda \bar{\lambda}$ (as algebra objects), which by Theorem \ref{alg_iso} happens iff $\lambda  \cong \alpha \kappa \beta$, where $\beta$ is an object in the right category. But the action of invertible objects is known - it is contained in the fusion data. So the list of planar algebras can be read off the list of principal graphs above - there is a distinct planar algebra corresponding to each equivalence class of odd vertices, where equivalence is given by the combined left and right action of invertible objects.

\end{proof}


\newcommand{\urlprefix}{}
\bibliographystyle{alpha}
\bibliography{bibliography}

\end{document}